\documentclass[12pt]{amsart}
\usepackage{amssymb,amsthm,amsmath,a4}
\newtheorem{thm}{Theorem}[section]
\newtheorem{theorem}[thm]{Theorem}
\newtheorem{corollary}[thm]{Corollary}
\newtheorem{lemma}[thm]{Lemma}

\theoremstyle{definition}

\newtheorem{example}[thm]{Example}

\theoremstyle{remark}
\newtheorem{remark}[thm]{Remark}

\numberwithin{equation}{section}

\newcommand{\al}{\alpha}

\newcommand{\AC}{\mathcal{A}}

\newcommand{\BR}{\mathrm{B}}
\newcommand{\BC}{\mathcal{B}}
\newcommand{\CC}{\mathcal{C}}
\newcommand{\SC}{\mathcal{S}}

\newcommand{\de}{\mathrm{\delta}}

\newcommand{\DC}{\mathcal{D}}

\newcommand{\eps}{\varepsilon}

\newcommand{\fR}{\mathrm f}
\newcommand{\ga}{\gamma}

\newcommand{\la}{\lambda}
\newcommand{\La}{\Lambda}
\newcommand{\LC}{\mathcal{L}}

\newcommand{\nR}{\mathrm{n}}

\newcommand{\Om}{\Omega}

\newcommand{\RC}{\mathcal{R}}

\newcommand{\sR}{\mathrm{s}}

\newcommand{\TC}{\mathcal{T}}
\newcommand{\Si}{\Sigma}
\newcommand{\si}{\sigma}

\newcommand{\uR}{\mathrm{u}}

\newcommand{\zbf}{\mathbf{z}}

\newcommand{\C}{\mathbb{C}}
\newcommand{\R}{\mathbb{R}}
\newcommand{\Sbb}{\mathbb{S}}

\newcommand{\Arg}{\operatorname{Arg}}
\newcommand{\diag}{\operatorname{diag}}
\newcommand{\diam}{\operatorname{diam}}
\newcommand{\dist}{\operatorname{dist}}

\newcommand{\ess}{{\mathrm{ess}}}

\newcommand{\rank}{\operatorname{rank}}
\newcommand{\re}{\operatorname{Re}}
\newcommand{\im}{\operatorname{Im}}

\newcommand{\Tr}{\operatorname{Tr}}
\newcommand{\Spec}{\operatorname{Spec}}

\def\leq {\leqslant}
\def\geq {\geqslant}

\begin{document}

\title[On the relation between an operator and its self-commutator]
{On the relation between an operator\\ and its self-commutator}

\author{N. Filonov}
\address{Steklov Mathematical Institute, Fontanka 27, St Petersburg, Russia}
 \email{filonov@pdmi.ras.ru}

\author{Y. Safarov}
\address{Department of Mathematics, King's College London,
Strand, London, UK} \email{yuri.safarov@kcl.ac.uk}

\date{October 2009}

\thanks{The research was supported by the EPSRC grant GR/T25552/01.}


\maketitle

Let $\,X\,$ and $\,Y\,$ be bounded self-adjoint operators on a
Hilbert space $\,H\,$. The paper deals with the following well known
problem: if the commutator $\,[X,Y\,]$ is small in an appropriate
sense, is there a pair of commuting operators $\,\tilde X\,$ and
$\,\tilde Y\,$ which are close to $\,X\,$ and $\,Y\,$? Note that for
general bounded operators $\,X\,$ and $\,Y\,$ it is not necessarily
true (see Subsection \ref{B8}).

For self-adjoint $\,X\,$ and $\,Y\,$, taking $\,A:=X+iY\,$, one can
reformulate the question as follows: if the self-commutator
$\,[A^*,A]\,$ is small, is there a normal operator $\,\tilde A\,$
close to $\,A\,$? There are some positive results in this direction.
Probably, the most famous is the Brown--Douglas--Fillmore theorem
\cite{BDF}.

\begin{theorem}\label{thm:bdf}
If $\,H\,$ is separable, $\,[A^*,A]\,$ is compact and the
corresponding to $\,A\,$ element of the Calkin algebra has trivial
index function then there is a normal operator $\,T\,$ such that
$\,A-T\,$ is compact.
\end{theorem}

\noindent Another is due to Huaxin Lin \cite{L3}.

\begin{theorem}\label{thm:huaxin}
There exists a continuous function $\,F:[0,\infty)\mapsto[0,1]\,$
vanishing at the origin such that the distance from $\,A\,$ to the
set of normal operators is estimated by $\,F(\|[A^*,A]\|)$ for all
finite rank operators $\,A\,$ with $\,\|A\|\leq1\,$.
\end{theorem}

A related question is whether an operator $\,A\,$ with small
self-commutator is close to a diagonal operator. Recall that an
operator $\,T\,$ on a separable Hilbert space is said to be diagonal
if it is represented by a diagonal matrix in some orthonormal basis.
Clearly, all diagonal operators are normal. The following result was
obtained in \cite{Be} and is usually referred to as the Weyl--von
Neumann--Berg theorem.

\begin{theorem}\label{thm:berg}
Let $\,A\,$ be a (not necessarily bounded) normal operator on a
separable Hilbert space. Then for each $\,\eps>0\,$ there exist a
diagonal operator $\,D_\eps\,$ and a compact operator $\,K_\eps\,$
with $\,\|K_\eps\|\leq\eps\,$ such that $\,A=D_\eps+K_\eps\,$.
\end{theorem}

Our main result is Theorem \ref{thm:main}, which shows that a
bounded operator $\,A\,$ belongs to a certain set associated with
its self-commutator whenever $\,A-\la I\,$ can be approximated by
invertible operators for all $\,\la\in\C\,$. Theorem \ref{thm:main}
implies both the BDF and Huaxin Lin's theorems. Moreover, it allows
us to refine the former and to extend the latter to operators of
infinite rank and other norms (see Subsections \ref{S3.matrices} and
\ref{S3.bdf}, Corollary \ref{cor:main-2} and Remark
\ref{rem:normal-2}). In particular, we obtain
\begin{enumerate}
\item[(i)]
a quantitative version of Theorem \ref{thm:bdf}, which links the BDF
and Weyl--von Neumann--Berg theorems for bounded operators,  and
\item[(ii)] an analogue of Huaxin Lin's theorem for the Schatten norms of finite
matrices.
\end{enumerate}

Theorem \ref{thm:main} holds for any unital $C^*$-algebra $L$ of
real rank zero, but in the general case we need a slightly stronger
condition on $\,A\,$. Namely, we assume that $\,A-\la I\,$ belong to
the closure of the connected component of unity in the set of
invertible elements in $\,L\,$ for all $\,\la\in\C\,$.

Our proof of Theorem \ref{thm:main} uses the $\,C^*$-algebra
technique developed in \cite{FR1} and \cite{FR2} and cannot be
significantly simplified by assuming that $\,L\,$ is the
$\,C^*$-algebra of all bounded operators. One of the main
ingredients in the proof is the part of Corollary \ref{cor:FR2}
which says that a normal operator satisfying the above condition can
be approximated by normal operators with finite spectra. This
statement is contained in \cite[Theorem 3.2]{FR2}. The authors
indicated how it could be proved but did not present complete
arguments. Therefore, for reader's convenience, we give
operator-theoretic proofs of this and relevant results. More
precisely, we deduce Corollary \ref{cor:FR2} from Theorem
\ref{thm:partition}, which seems to be new and may be of independent
interest.

\section{Notation and auxiliary results}\label{S1}

\subsection{Notation and definitions}\label{S1.1}
Let $\,H\,$ be a complex Hilbert space (not necessarily separable).
Throughout the paper,
\begin{enumerate}
\item[$\bullet$]
$\BC(H)\,$ is the $C^*$-algebra of all bounded operators in $\,H\,$;
\item[$\bullet$]
$\si(A)\,$ denotes the spectrum of $\,A\in\BC(H)\,$;
\item[$\bullet$]
$\,L\,$ is a unital $\,C^*$-algebra represented on the Hilbert space
$\,H\,$, so that $\,L\subseteq \BC(H)\,$.
\end{enumerate}
Recall that, by the Gelfand--Naimark theorem, such a representation
exists for any unital $\,C^*$-algebra $\,L\,$.

Every operator $\,A\in \BC(H)\,$ admits the polar decomposition
$\,A=V|A|\,$, where $\,|A|\,$ is the self-adjoint operator
$\,\sqrt{A^*A}\,$ and $\,V\,$ is an isometric operator such that
$\,VH=\overline{AH}\,$ and $\,\ker V=\ker A=\ker|A|\,$. If $\,A\,$
is normal then $\,V|A|=|A|V\,$.

\begin{remark}\label{rem:l-1}
If $\,A\in L\,$ then $\,f(|A|)\in L\,$ for any continuous function
$\,f\,$. If, in addition, $\,A\,$ is invertible then
$\,V=A|A|^{-1}\,$ is a unitary element of $\,L\,$ because
$\,|A|^{-1}\,$ can be approximated by continuous functions of
$\,|A|\,$. In particular, this implies that $\,A^{-1}\in L\,$. If
$\,A\in L\,$ is not invertible then the isometric operator $\,V\,$
in its polar representation does not have to belong to $\,L\,$.
\end{remark}

\begin{remark}\label{rem:projections}
The spectral projections of a normal operator $\,A\in L\,$ may not
lie in $\,L\,$. However, the spectral projection corresponding to a
connected component of $\,\si(A)\,$ belongs to $\,L\,$, since it can
be written as a continuous function of $\,A\,$.
\end{remark}

Further on
\begin{enumerate}
\item[$\bullet$]
$L^{-1}\,$ is the set of invertible operators in $L$;
\item[$\bullet$]
$L_0^{-1}\,$ denotes the connected component of $L^{-1}\,$
containing the identity operator;
\item[$\bullet$]
$\,L_\nR\,$, $\,L_\uR\,$, and $\,L_\sR\,$ are the sets of normal,
unitary and self-adjoint operators in $L$ respectively;
\item[$\bullet$]
$\,L_\fR\,$ is the set of operators $\,A\in L\,$ with finite
spectra;
\item[$\bullet$]
if $\,M\subset L\,$ then $\,\overline M\,$ denotes the norm closure
of the set $\,M\,$ in $\,L\,$.
\end{enumerate}
Clearly, the sets $\,L^{-1}\,$ and $\,L_0^{-1}\,$ are open in
$\,L\,$, and the sets $\,L_\nR\,$, $\,L_\uR\,$ and $\,L_\sR\,$ are
closed.

One says that
\begin{enumerate}
\item[$\bullet$]
$\,L\,$ has {\it real rank zero\/} if $\,\overline{L^{-1}\cap
L_\sR}=L_\sR\,$.
\end{enumerate}
The concept of real rank of a $\,C^*$-algebra was introduced in
\cite{BP}. A unital $\,C^*$-algebra $\,L\,$ has real rank zero if
and only if any self-adjoint operator $\,A\in L\,$ is the norm limit
of a sequence of self-adjoint operators from $\,L\,$ with finite
spectra (see Corollary \ref{cor:FR1} and Subsection \ref{B1}).

\begin{remark}\label{rem:vonNeumann}
Note that any self-adjoint operator $\,A\in\BC(H)\,$ is approximated
in the norm topology by invertible self-adjoint operators of the
form $\,f(A)\,$, where $\,f\,$ are suitable real-valued Borel
functions. Therefore, all von Neumann algebras (in particular,
$\,\BC(H)\,$ and the algebra of finite $\,m\times m$-matrices) have
real rank zero.
\end{remark}

\begin{example}\label{ex:real-rank}
The minimal unital $\,C^*$-algebra $\,L_A\,$ containing a given
bounded self-adjoint operator $\,A\,$ consists of normal operators
$\,f(A)\,$, where $\,f\,$ are continuous complex-valued functions on
$\,\si(A)\,$. If there is an open interval $\,(a,b)\subset\si(A)\,$
then $\,A-\frac{a+b}2\,I\not\in\overline{L^{-1}\cap L_\sR}\,$ and,
consequently, $\,L_A\,$ is not of real rank zero.
\end{example}

Our main results hold for $\,C^*$-algebras of real rank zero and
operators $\,A\in L\,$ satisfying the following condition
\begin{enumerate}
\item[{\bf(C)}]
$\,A-\la I\in\overline{L_0^{-1}}\,$ for all $\,\la\in\C\,$.
\end{enumerate}

\subsection{Auxiliary lemmas}\label{S1.2}
We shall need the following simple lemmas.

\begin{lemma}\label{lem:l0-1}
Let $\,\widehat{L}\,$ be the subset of the direct product
$\,\C\times L\,$ which consists of all pairs $\,(\la,A)\,$ such that
$\,\la\not\in\si(A)\,$. If $\,A_0-\la_0I\in L_0^{-1}\,$ for some
$\,(\la_0,A_0)\in\widehat{L}\,$ then $\,A-\la
I\in\overline{L_0^{-1}}\,$ for all $\,(\la,A)\,$ lying in the
closure of the connected component of $\,\widehat{L}\,$ that
contains $\,(\la_0,A_0)\,$.
\end{lemma}

\begin{proof}
Let $\,\widehat{L}_0\,$ be the connected component of
$\,\widehat{L}\,$ containing $\,(\la_0,A_0)\,$. Since the set
$\,\widehat{L}\,$ is open, $\,\widehat{L}_0\,$ is path-connected. If
$\,(\la,A)\in\widehat{L}_0\,$ and
$\,(\la_t,A_t)\subset\widehat{L}_0\,$ is a path in
$\,\widehat{L}_0\,$ from $\,(\la_0,A_0)\,$ to $\,(\la,A)\,$ then
$\,A_t-\la_tI\in L^{-1}\,$ for all $\,t\,$, which implies that
$\,A-\la I\in L_0^{-1}\,$. If $\,(\la,A)\,$ belongs to the closure
of $\,\widehat{L}_0\,$ then $\,A-\la I\,$ can be approximated by
operators $\,A_n-\la_nI\in L_0^{-1}\,$ with
$\,(\la_n,A_n)\in\widehat{L}_0\,$.
\end{proof}

\begin{remark}\label{rem:l0-1}
If $\,|\mu|>\|A\|\,$ then $\,A-\mu I\in L_0^{-1}\,$ because
$\,[0,1]\ni t\mapsto tA-\mu I\,$ is a path in $\,L^{-1}\,$ from
$\,-\mu I\in L_0^{-1}\,$ to $\,A-\mu I\,$. Therefore, Lemma
\ref{lem:l0-1} implies that $\,A\,$ satisfies the condition {\bf(C)}
whenever $\,\C\setminus\si(A)\,$ is a dense connected subset of
$\,\C\,$. In particular, {\bf(C)} is fulfilled for all compact
operators $\,A\in L\,$, all $\,A\in L_\sR\cup L_\fR\,$, and all
unitary operators $\,A\in L_\uR\,$ whose spectra do not coincide
with the whole unit circle.
\end{remark}

\begin{lemma}\label{lem:l0-2}
Let $\,A\in L^{-1}\,$ and $\,A=U|A|\,$. Then $\,A\in L_0^{-1}\,$ if
and only if $\,U\in L_0^{-1}\cap L_\uR\,$.
\end{lemma}

\begin{proof}
$\,[0,1]\ni t\mapsto U(tI+(1-t)|A|)\,$ is a path in $\,L^{-1}\,$
from $\,A\,$ to $\,U\,$.
\end{proof}

\begin{lemma}\label{lem:l0-3}
Assume that $\,L\,$ has real rank zero. Then $\,U\in L_\uR\cap
L_0^{-1}\,$ if and only if for every $\,\eps\in(0,1)\,$ there exist
unitary operators $\,U_\eps,W_\eps\in L_\uR\,$ such that
$\,U=W_\eps\,U_\eps\,$, $\,-1\not\in\si(U_\eps)\,$ and
$\,\|W_\eps-I\|\leq\eps\,$.
\end{lemma}

\begin{proof}
Recall that the point $-1$ does not belong to the spectrum of
$\,U\in L_\uR\,$ if and only if $\,U\,$ is the Cayley transform of a
self-adjoint operator $\,X\,$, that is, $\,U=(iI-X)^{-1}(iI+X)\,$
where $\,X=i\,(U+I)^{-1}(U-I)\in L_\sR\,$. For every such an
operator $\,U\,$, the principal branch of the argument $\,\Arg\,$ is
continuous in a neighbourhood of $\,\si(U)\,$, so that $\,\Arg U\in
L_\sR\,$ and $\,\exp(it\,\Arg U)\,$ is a path in $\,L_\uR\cap
L_0^{-1}\,$ from $\,I\,$ to $\,U\,$.

Assume first that $\,U=W_\eps\,U_\eps\,$ where $\,U_\eps\,$ and
$\,W_\eps\,$ satisfy the above conditions. Then $\,U\in L_\uR\,$,
$\,-1\not\in\si(W_\eps)\,$ and $\,\exp(it\,\Arg
W_\eps)\,\exp(it\,\Arg U_\eps)\,$ is a path in $\,L_\uR\cap
L_0^{-1}\,$ from $\,I\,$ to $\,U\,$. Thus $\,U\in L_\uR\cap
L_0^{-1}\,$.

Assume now that $\,U\in L_\uR\cap L_0^{-1}\,$. Then there exists a
path $\,Z(t)\subset L_0^{-1}\,$ from $\,I\,$ to $\,U\,$. The
``normalized'' path $\,\tilde Z(t)=Z(t)|Z(t)|^{-1}\,$ lies in
$\,L_\uR\cap L_0^{-1}\,$ and also joins $\,I\,$ and $\,U\,$. Let us
choose a finite collection of points $\,0=t_0<t_1<t_2<\dots<t_m=1\,$
such that $\,\|\tilde Z(t_j)-\tilde Z(t_{j-1})\|<1\,$ and define
$\,V_j:=\tilde Z(t_j)\,\tilde Z^{-1}(t_{j-1})\,$. Then
$\,U=V_m\,V_{m-1}\,\ldots V_1\,$ and $\,\|V_j-I\|<1\,$, so that
$\,-1\not\in\si(V_j)\,$ for all $\,j\,$.

Let $\,V=(iI-Y)^{-1}(iI+Y)\,$ and $\,\tilde V=(iI-\tilde
Y)^{-1}(iI+\tilde Y)\,$ where $\,Y,\tilde Y\in L_\sR\,$. Since
$\,L\,$ has real rank zero, for each $\,\de>0\,$ we can find
$\,Y_\de\in L_\sR\cap L^{-1}\,$ such that $\,\|Y-Y_\de\|<\de\,$, and
then $\,\tilde Y_\de\in L_\sR\,$ such that $\,\|\tilde Y-\tilde
Y_\de\|<\de\,$ and $\,\tilde Y_\de-Y_\de^{-1}\in L^{-1}\,$. If
$\,V_\de=(iI-Y_\de)^{-1}(iI+Y_\de)\,$ and $\,\tilde V_\de=(iI-\tilde
Y_\de)^{-1}(iI+\tilde Y_\de)\,$ then $\,\|V_\de\tilde V_\de-V\tilde
V\|\to0\,$ as $\,\de\to0\,$ because the function
$\,t\mapsto(iI-t)^{-1}(iI+t)\,$ is continuous, and
$\,-1\not\in\si(V_\de\tilde V_\de)\,$ because
\begin{multline*}
(iI-Y_\de)\,(V_\de\tilde V_\de+I)\,(iI-\tilde Y_\de)\ =\
(iI+Y_\de)\,(iI+\tilde Y_\de)
+(iI-Y_\de)(iI-\tilde Y_\de)\\
=\ 2\,(Y_\de\tilde Y_\de-I)=2\,Y_\de\,(\tilde Y_\de-Y_\de ^{-1})\
\in\ L^{-1}\,.
\end{multline*}

Thus we see that the composition of two unitary operators whose
spectra do not contain $\,-1\,$ can be approximated by unitary
operators with the same property. By induction, the same is true for
the composition of any finite collection of unitary operators. In
particular, there exists $\,U_\eps\in L_\uR\cap L_0^{-1}\,$ such
that $\,-1\not\in\si(U_\eps)\,$ and $\,\|U-U_\eps\|<\eps\,$. Taking
$\,W_\eps:=U\,U_\eps^{-1}\,$, we obtain the required representation
of $\,U\,$.
\end{proof}

\begin{lemma}\label{lem:uniform}
Assume that $\,L\,$ has real rank zero. Then
\begin{enumerate}
\item[(1)]
for every $\,A\in\overline{L_0^{-1}}\,$ and every $\,\de>0\,$ there
exists an operator $\,S_\de\in L_0^{-1}\,$ such that
$\,\|S_\de^{-1}\|\leq\de^{-1}\,$ and $\,\|A-S_\de\|\leq2\de\,$.
\item[(2)]
If $\,A\in L_\sR\,$ then one can find a self-adjoint operator
$\,S_\de\in L^{-1}\,$ satisfying the above conditions.
\item[(3)]
For every $\,S\in L_0^{-1}\,$ there exists a continuous path
$\,Z:[0,1]\mapsto L_0^{-1}\,$ such that $\,Z(0)=I\,$, $\,Z(1)=S\,$,
$\,\|Z(t)^{-1}\|\leq\max\left\{1\,,\,\|S^{-1}\|\right\}\,$ for all
$\,t\in[0,1]\,$ and
\begin{equation}\label{path-1}
\|Z(t)-Z(r)\|\ \leq\
|t-r|\,(1+2\pi)\,\max\left\{1\,,\,\|S\|\right\}\,,\qquad\forall
t,r\in[0,1]\,.
\end{equation}
\end{enumerate}
\end{lemma}

\begin{proof}\

(1) Let us choose an arbitrary operator $\,B\in L_0^{-1}\,$ such
that $\,\|A-B\|\leq\de\,$, and let $\,B=V|B|\,$ be its polar
decomposition. Then, by Lemma \ref{lem:l0-2}, we have $\,V\in
L_0^{-1}\cap L_\uR\,$ and $\,S_\de:=V((|B|-\de I)_++\de I)\in
L_0^{-1}\,$. Obviously, $\,\|S_\de^{-1}\|=\|((|B|-\de I)_++\de
I)^{-1}\|\leq\de^{-1}\,$ and $\,\|B-S_\de\|=\|(|B|-\de I)_++\de
I-|B|\,\|\leq\de\,$, so that $\,\|A-S_\de\|\leq2\de\,$.

\smallskip
(2) If $\,A\in L_\sR\,$ then there exists an operator $\,B=V|B|\in
L_0^{-1}\cap L_\sR\,$ such that $\,\|A-B\|\leq\de\,$. As in (1), we
can take $\,S_\de:=V((|B|-\de I)_++\de I)\in L_\sR\,$.

\smallskip
(3) Let $\,S:=U|S|\,$ be the polar representation of $\,S\,$. By
Lemma \ref{lem:l0-2}, $\,U\in L_0^{-1}\cap L_\uR\,$. Therefore
$\,A:=W_\eps\,U_\eps\,|S|\,$, where $\,W_\eps\,$ and $\,U_\eps\,$
are unitary operators satisfying the conditions of Lemma
\ref{lem:l0-3}.

Let $\,Z_1(t):=\exp(it\Arg W_\eps)\,$, $\,Z_2(t):=\exp(it\Arg
U_\eps)\,$ and $\,Z_3(t):=t|A_\de|+(1-t)I\,$, where $\,t\in[0,1]\,$.
Each $\,Z_j(t)\,$ is a path in $\,L_0^{-1}\,$, and so is
$\,Z(t):=Z_1(t)Z_2(t)Z_3(t)\,$. Obviously, $\,Z(0)=I\,$,
$\,Z(1)=S\,$ and
$$
\|Z(t)^{-1}\|\ =\ \|(t|S|+(1-t)I)^{-1}\|\ =\
\left(t\|S^{-1}\|^{-1}+(1-t)\right)^{-1}\ \leq\
\max\left\{1\,,\,\|S^{-1}\|\right\}\,.
$$

One can easily see that
$\,\|Z_3(t)-Z_3(r)\|\leq|t-r|\,\max\left\{1\,,\,\|S\|\right\}\,$,
$\,\|Z_3(t)\|\leq\max\left\{1\,,\,\|S\|\right\}\,$ and
$\,\|Z_1(t)\|=\|Z_2(t)\|=1\,$. Since
$\,|e^{it\theta}-e^{ir\theta}|\leq\pi\,|t-r|\,$ for all
$\,r,t\in\R\,$ and $\,\theta\in(-\pi,\pi)\,$, we also have
$\,\|Z_j(t)-Z_j(r)\|\leq\pi\,|t-r|\,$ for $\,j=1,2\,$. These
inequalities imply \eqref{path-1}.
\end{proof}

\section{Main results}\label{S2}

\subsection{Resolution of the identity}\label{S2.1}
The following theorem will be proved in Appendix \ref{A}. Roughly
speaking, it says that a normal operator $\,A\in L_\nR\,$ satisfying
the condition {\bf(C)} has a resolution of the identity in $\,L\,$
associated with any finite open cover of $\,\si(A)\,$.

\begin{theorem}\label{thm:partition}
Assume that $\,L\,$ has real rank zero. Let $\,A\in L_\nR\,$, and
let $\,\{\Om_j\}_{j=1}^m\,$ be a finite open cover of $\,\si(A)\,$.
If $\,A\,$ satisfies the condition {\bf(C)} then there exists a
family of mutually orthogonal projections $\,P_j\in L\,$ such that
\begin{equation}\label{projections}
\sum_{j=1}^mP_j=I\quad\text{and}\quad P_jH\subset\Pi_{\Om_j}H\
\text{for all}\ j=1,\ldots,m\,,
\end{equation}
where $\,\Pi_{\Om_j}\,$ are the spectral projections of $\,A\,$
corresponding to the sets $\,\Om_j\,$.
\end{theorem}

\begin{remark}\label{rem:von-Neumann}
The operators $\,P_j\,$ can be thought of as approximate spectral
projections of $\,A\,$. If $\,L\,$ is a von Neumann algebra then the
spectral projections of $\,A\,$ belong to $\,L\,$ and one can simply
take $\,P_j=\Pi_{\Om'_j}\,$, where $\,\{\Om'_j\}_{j=1}^m\,$ is an
arbitrary collection of mutually disjoint subsets
$\,\Om'_j\subset\Om_j\,$ covering $\,\si(A)\,$. However, even in
this situation Theorem \ref{thm:partition} may be useful, since the
projections $\,P_j\,$ constructed in the proof continuously depend
on $\,A\,$ in the norm topology.
\end{remark}

The following simple lemma shows how Theorem \ref{thm:partition} can
be applied for approximation purposes.

\begin{lemma}\label{lem:partition}
Let $\,A\in L_\nR\,$, and let $\,\{\Om_j\}_{j=1}^m\,$ be a finite
open cover of $\,\si(A)\,$ whose multiplicity does not exceed
$\,k\,$. If there exist mutually orthogonal projections $\,P_j\,$
satisfying \eqref{projections} then
\begin{equation}\label{partition}
\|A-\sum_{j=1}^mz_j\,P_j\|\ \leq\ \sqrt k\,\max_j(\diam\,\Om_j)
\end{equation}
for any collection of points $\,z_j\in\Om_j\,$.
\end{lemma}

\begin{proof}
If $\,z_j\in\Om_j\,$ then
\begin{multline*}
\|Au-\sum_{j=1}^mz_j\,P_ju\|^2\ =\
\sum_{j=1}^m\|P_jAu-z_j\,P_ju\|^2\ \leq\
\sum_{j=1}^m\|\Pi_{\Om_j}(Au-z_ju)\|^2\\
=\ \sum_{j=1}^m\|(A-z_jI)\,\Pi_{\Om_j}u\|^2\ \leq\
\sum_{j=1}^m\|\Pi_{\Om_j}u\|^2\,(\diam\,\Om_j)^2\ \leq\
k\,\|u\|^2\,\max_j(\diam\,\Om_j)^2
\end{multline*}
for all $\,u\in H\,$. Taking the supremum over $\,u\,$, we obtain
\eqref{partition}.
\end{proof}

Theorem \ref{partition} and Lemma \ref{lem:partition} imply the
following corollaries.

\begin{corollary}\label{cor:FR1}
The following statements are equivalent.
\begin{enumerate}
\item[(1)]
A $\,C^*$-algebra $\,L\,$ has real rank zero.
\item[(2)]
Every self-adjoint operator $\,A\in L_\sR\,$ has approximate
spectral projections in the sense of Theorem
{\rm\ref{thm:partition}}, associated with any finite open cover of
its spectrum.
\item[(3)] $\,L_\sR=\overline{L_\fR\cap L_\sR}\,$.
\end{enumerate}
\end{corollary}

\begin{corollary}\label{cor:FR2}
Assume that $\,L\,$ has real rank zero. Then for every normal
operator $\,A\in L_\nR\,$ the following statements are equivalent.
\begin{enumerate}
\item[(1)]
The operator $\,A\,$ satisfies the condition {\bf(C)}.
\item[(2)]
The operator $\,A\,$ has approximate spectral projections in the
sense of Theorem {\rm\ref{thm:partition}}, associated with any
finite open cover of its spectrum.
\item[(3)] $\,A\in\overline{L_\fR\cap L_\nR}\,$.
\end{enumerate}
\end{corollary}

\begin{proof} The corollaries are proved in the same manner.

By Remark \ref{rem:l0-1}, every self-adjoint operator $\,A\in
L_\sR\,$ satisfies the condition {\bf(C)}. Therefore the
implications $\,(1)\Rightarrow(2)\,$ follow from Theorem
\ref{thm:partition}.

Any subset of $\,\C\,$ admits a cover $\,\{\Om_j\}_{j=1}^m\,$ of
multiplicity four by open squares $\,\Om_j\,$ of arbitrarily small
size. If $\,A\in L_\nR\,$ has approximate spectral projections
$\,P_j\,$ associated with all such covers of its spectrum then, in
view of \eqref{partition}, the operator $\,A\,$ can be approximated
by operators of the form $\,\sum_{j=1}^mz_j\,P_j\in L_\fR\cap
L_\nR\,$. Moreover, if $\,A\in L_\sR\,$ then we can take
$\,z_j\in\R\,$, so that $\,\sum_{j=1}^mz_j\,P_j\in L_\fR\cap
L_\sR\,$. Thus $\,(2)\Rightarrow(3)\,$.

Finally, in view of Remark \ref{rem:projections}, every operator
$\,T\in L_\fR\cap L_\nR\,$ can be written in the form
$\,\sum_{j=1}^mz_j\,\Pi_j\,$, where $\,z_j\in\R\,$ whenever $\,T\in
L_\sR\,$ and $\,\Pi_j\,$ are mutually orthogonal projections lying
in $\,L\,$. If $\,A\,$ is approximated by a sequence of such
operators then it can also be approximated by a sequence of
operators $\,\sum_{j=1}^m\tilde z_j\,\Pi_j\,$, where $\,\Pi_j\,$ are
the same projections, $\,\tilde z_j\ne0\,$ and $\,\im\tilde z_j=\im
z_j\,$. This shows that $\,(3)\Rightarrow(1)\,$.
\end{proof}

\begin{remark}\label{rem:disclaimer}
The implications $\,(1)\Leftrightarrow(3)\,$ in the above
corollaries are known results (see Subsection \ref{B1} and
\cite[Theorem 3.2]{FR2}). In \cite{FR2}, the authors explained that
the part $\,(1)\Rightarrow(3)\,$ of Corollary \ref{cor:FR2} would
follow from the existence of projections `that approximately commute
with $\,A\,$ and approximately divide $\,\si(A)\,$ into disjoint
components'. The implications $\,(1)\Rightarrow(2)\Rightarrow(3)\,$
in Corollary \ref{cor:FR2} give a precise meaning to their
statement.
\end{remark}

\subsection{The main theorem}\label{S2.2}
In this subsection
\begin{enumerate}
\item[$\bullet$]
$\,\BR(r):=\{T\in L\,:\,\|T\|\leq r\}\,$ is the closed ball about
the origin in $\,L\,$ of radius $\,r\,$;
\item[$\bullet$]
$\,M_T\,$ denotes the convex hull of the set
$\,\bigcup_{S_1,S_2\in\BR(1)}S_1TS_2\,$, where $\,T\in L\,$;
\item[$\bullet$]
$\,J_T\,$ is the two-sided ideal in $\,L\,$ generated by the
operator $\,T\in L\,$.
\end{enumerate}

\begin{remark}\label{rem:M-ideal}
The ideal $\,J_T\,$ consists of finite linear combinations of
operators of the form $\,S_1TS_2\,$ where $\,S_1,S_2\in L\,$.
Therefore $\,J_T=\bigcup_{t\geq0}tM_T\,$ and $\,M_T\subset
J_T\cap\BR(\|T\|)\,$.
\end{remark}

\begin{remark}\label{rem:M-unitary}
The unit ball $\,\BR(1)\,$ coincides with the closed convex hull of
$\,L_\uR\,$ (see, for example, \cite{RD}). This implies that
$\,M_T\,$ is a subset of the closed convex hull of the set
$\,\bigcup_{U,V\in L_\uR}UTV\,$. Moreover, if
$\,\overline{L^{-1}}=L\,$ then $\,\BR(1)\,$ coincides with the
convex hull of $\,L_\uR\,$ (see \cite{R}) and, consequently, every
element of $\,M_T\,$ is a finite convex combination of operators of
the form $\,UTV\,$ with $\,U,V\in L_\uR\,$.
\end{remark}

We shall say that a continuous real valued function $\,f\,$
satisfies the condition (C$_{\eps,r}$) for some $\,\eps,r>0\,$ if
$\,f\,$ is defined on the interval $\,[-r-\eps,r+\eps]\,$  and
\begin{enumerate}
\item[(C$_{\eps,r}$)]
there exists $\,\eps'\in(0,\eps)\,$ such that the set
$\,\{x\in\R:f(x)=y\,\}\,$ is an $\,\eps'$-net in
$\,\{x\in\R:x^2+y^2\leq(r+\eps)^2\,\}\,$ for each
$\,y\in[-r-\eps,r+\eps]\,$.
\end{enumerate}
The condition (C$_{\eps,r}$) is fulfilled whenever the function
$\,f\,$ sufficiently rapidly oscillates between $\,-r-\eps\,$ and
$\,r+\eps\,$. In particular, it holds for $\,f(x)=(r+\eps)\cos(\pi
x/\eps)\,$.

\begin{lemma}\label{lem:xy}
Let $\,L\,$ have real rank zero, and let $\,X,Y\in L_\sR\,$ be
self-adjoint operators such that  the operator $\,A:=X+iY\,$
satisfies the condition {\bf(C)}. Then, for every function $\,f\,$
satisfying the condition {\rm(C$_{\eps,r}$)} with $\,r=\|A\|\,$, the
operator $\,Y\,$ belongs to the closure of the set
$\,f\left(X+\BR(\eps)\cap L_\sR\right)+J_{[X,Y]}\cap L_\sR\,$.
\end{lemma}

\begin{proof}
Assume first that $\,J_{[X,Y]}=\{0\}\,$, so that $\,A\,$ is normal.
By Corollary \ref{cor:FR2}, for each $\,\de\in(0,\eps]\,$ there
exists an operator $\,A_\de\in L_\nR\cap L_\fR\,$ with finite
spectrum $\,\si(A_\de)=\{z_1,\ldots,z_m\}\,$ such that $\,\|A-
A_\de\|\leq\de\,$ and, consequently, $\,|z_j|\leq r+\de\,$ for all
$\,j\,$. In view of (C$_{\eps,r}$), one can find real numbers
$\,\eps_k\in[-\eps',\eps']\,$ such that $\,z_k+\eps_k\,$ lie on the
graph of $\,f\,$ for all $\,k=1,\ldots,m\,$. Let $\,A'_\de\,$ be the
operator with eigenvalues $\,z_k+\eps_k\,$ and the same spectral
projections as $\,A_\de\,$. Then $\,\|Y-\im A'_\de\|\leq\de\,$,
$\,\|X-\re A'_\de\|\leq\de+\eps'\,$ and $\,\im A'_\de=f(\re
A'_\de)\,$. This implies that
$$
\left(Y+\BR(\de)\right)\bigcap f\left(X+\BR(\eps)\cap L_\sR\right)\
\ne\ \varnothing\,,\qquad\forall\de\in(0,\eps-\eps']\,.
$$
Letting $\,\de\to0\,$, we see that $\,Y\in
\overline{f\left(X+\BR(\eps)\cap L_\sR\right)}\,$.

Assume now that $\,J_{[X,Y]}\ne\{0\}\,$ and denote
$\,L':=\overline{J_{[X,Y]}}\,$. Let us consider the quotient
$\,C^*$-algebra $\,L/L'\,$ and the corresponding quotient map
$\,\pi: L\mapsto L/L'\,$. Since the map $\,\pi\,$ is continuous and
$\,\pi S=\pi(S+S^*)/2\,$ for all self-adjoint elements $\,\pi S\in
L/L'\,$, the quotient algebra also has real rank zero and $\,\pi
S\in\overline{(L/L')_0^{-1}}\,$ whenever
$\,S\in\overline{L_0^{-1}}\,$. The latter implies that the normal
element $\,\pi A\,$ of the quotient algebra $\,L/L'\,$ also
satisfies the condition {\bf(C)}.

Applying the previous result with $\,\eps\,$ replaced by
$\,\eps_0\in(\eps',\eps)\,$ to $\,\pi A=\pi X+i\pi Y\,$, we can find
a sequence of operators $\,X_n\in L_\sR\,$ such that $\,f(\pi
X_n)\to \pi Y\,$ as $\,n\to\infty\,$ and $\,\|\pi X-\pi
X_n\|\leq\eps_0\,$ for all $\,n\,$. Since $\,\|T\|\geq\|\re T\|\,$
for all $\,T\in L\,$, we have
\begin{equation}\label{xy1}
 \|\pi T\|\ :=\ \inf_{R\in L'}\|T-R\|\
=\ \inf_{R\in L'\cap L_\sR}\|T-R\|\,,\qquad\forall T\in L_\sR\,.
\end{equation}
Therefore, there exist operators $\,R_n\in L'\cap L_\sR\,$ such that
$\,X_n+R_n\in X+\BR(\eps)\cap L_\sR\,$. Since $\,f\,$ can be
uniformly approximated by polynomials on any compact subset of
$\,\R\,$ and $\,Q(\pi X_n)=\pi Q(X_n)\,$ for any polynomial $\,Q\,$,
we have $\,f(\pi X_n)=f(\pi(X_n+R_n))=\pi f(X_n+R_n)\,$ and,
consequently, $\,\|\pi(f(X_n+R_n)-Y)\|\to0\,$. In view of
\eqref{xy1}, there exist operators $\,\tilde R_n\in L'\cap L_\sR\,$
such that $\,f(X_n+R_n)+\tilde R_n\to Y\,$ as $\,n\to\infty\,$. This
implies that $\,Y\,$ belongs to the closure of the set
$\,f\left(X+\BR(\eps)\cap L_\sR\right)+L'\cap L_\sR\,$ which
coincides with $\,\overline{f\left(X+\BR(\eps)\cap
L_\sR\right)+J_{[X,Y]}\cap L_\sR}\,$.
\end{proof}

\begin{corollary}\label{cor:normal-1}
Assume that $\,L\,$ has real rank zero. If $\,A\in L\,$ satisfies
{\bf(C)} then
\begin{equation}\label{xy-2}
A\ \in\ \overline{\BR(\|A\|)\cap L_\nR+J_{[A^*,A]}\cap L_\sR}\,.
\end{equation}
\end{corollary}

\begin{proof}
Let $\,r:=\|A\|\,$. Given $\,\eps>0\,$, let us choose a function
$\,f_\eps\,$ satisfying the condition (C$_{\eps,r}$) whose graph
lies in the disc $\,\{x^2+y^2\leq(r+\eps)^2\}\,$. Applying Lemma
\ref{lem:xy}, we can find an operator $\,X_\eps\in\re
A+\BR(\eps)\cap L_\sR\,$ such that $\,\im A\in
f_\eps(X_\eps)+J_{[A^*,A]}\cap L_\sR+\BR(\eps)\cap L_\sR\,$. The
operator $\,\tilde A_\eps:=X_\eps+if_\eps(X_\eps)\,$ is normal, and
$\,A-\tilde A_\eps\in J_{[A^*,A]}\cap L_\sR+\BR(2\eps)\cap L_\sR\,$.
Therefore, there exist operators $\,R_\eps\in J_{[A^*,A]}\cap
L_\sR\,$ such that $\,\tilde A_\eps+R_\eps\to A\,$ as
$\,\eps\to0\,$. Since $\,x^2+(f(x))^2\leq(r+\eps)^2\,$, we have
$$
\|\tilde A_\eps\|^2\ =\ \|X_\eps\|^2+\|f_\eps(X_\eps)\|^2\ \leq\
(r+\eps)^2\,.
$$
If $\,A_\eps:=r\,(r+\eps)^{-1}\tilde A_\eps\,$ then, by the above,
$\,A_\eps\in\BR(r)\cap L_\nR\,$ and $\,A_\eps+R_\eps\to A\,$ as
$\,\eps\to0\,$.
\end{proof}

\begin{remark}\label{rem:normal-2}
By Corollary \ref{cor:normal-1},
$\,\overline{L_\nR+L'}=\overline{L_\nR+L'\cap L_\sR}\,$ for any
two-sided ideal $\,L'\subset L\,$ in a $C^*$-algebra $\,L\,$ of real
rank zero. Indeed, if $\,A\in L_\nR+L'\,$ then $\,J_{[A^*,A]}\subset
L'\,$ and, in view of \eqref{xy-2}, $\,A\,$ can be approximated by
operators from $\,L_\nR+L'\cap L_\sR\,$.
\end{remark}

The following refinement of Corollary \ref{cor:normal-1} is the main result of the paper.

\begin{theorem}\label{thm:main}
There is a nonincreasing function
$\,h:(0,\infty)\mapsto[0,\infty)\,$ such that $\,h(\eps)=0\,$ for
all $\,\eps\geq1\,$ and
\begin{equation}\label{main-1}
A\ \in\ \BR(\|A\|)\cap L_\nR\,+\,h(\eps)\,M_{[A^*,A]}\cap
L_\sR\,+\,\BR(\eps)
\end{equation}
for all $\,\eps\in(0,\infty)\,$, all $\,C^*$-algebras $L$ of real
rank zero and all operators $\,A\in\BR(1)\,$ satisfying the
condition {\bf(C)}.
\end{theorem}

\begin{remark}\label{rem:eps}
In other words, the inclusion \eqref{main-1} means that for each
$\,\eps>0\,$ there exist a normal operator $\,T(\eps)\in L_\nR\,$
and a finite linear combination
\begin{equation}\label{lin-comb1}
S(\eps)\ =\ \sum_jc(j,\eps)\,S_1(j,\eps)[A^*,A]S_2(j,\eps)
\end{equation}
with $\,S_k(j,\eps)\in L\,$ and $\,c(j,\eps)\in(0,1]\,$ such that
$\,\|T(\eps)\|\leq\|A\|\,$, $\,S(\eps)\in L_\sR\,$,
$\,\|S_k(j,\eps)\|\leq1\,$, $\,\sum_jc(j,\eps)=1\,$ and
$$
\|\,A-T(\eps)-h(\eps)\,S(\eps)\;\|\ \leq\ \eps\,.
$$
Note that \eqref{lin-comb1} can be written as a linear combination
of self-adjoint operators,
\begin{equation}\label{lin-comb2}
S(\eps)\ =\ \sum_jc(j,\eps)\left(S_+^*(j,\eps)[A^*,A]S_+(j,\eps)
-S_-^*(j,\eps)[A^*,A]S_-(j,\eps)\right)
\end{equation}
where $\,\|S_\pm(j,\eps)\|\leq1\,$. Indeed, if
$\,S_\pm(j,\eps):=\frac12\left(S_1^*(j,\eps)\pm
S_2(j,\eps)\right)\,$ then the real part of each term in the right
hand side of \eqref{lin-comb1} coincides with corresponding term in
\eqref{lin-comb2}.
\end{remark}

\begin{proof}
Let us consider a family of $\,C^*$-algebras $\,L_\xi\,$ of real
rank zero parameterised by $\,\xi\in\Xi\,$, where $\,\Xi\,$ is an
arbitrary index set, and let $\,\LC\,$ be their direct product. By
definition, the $\,C^*$-algebra $\,\LC\,$ consists of families
$\,\SC=\{S_\xi\}\,$ with $\,S_\xi\in L_\xi\,$ such that
$\,\|\SC\|_\LC:=\sup_{\xi\in\Xi}\|S_\xi\|<\infty\,$,
$\,\SC^*:=\{S_\xi^*\}\,$ and $\,\SC\tilde\SC=\{S_\xi\tilde
S_\xi\}\,$. Let $\,\BR_\LC(r)\,$ and $\,\BR_\xi(r)\,$ be the balls
of radius $\,r\,$ about the origin in $\,\LC\,$ and $\,L_\xi\,$
respectively.

In view of Lemma \ref{lem:uniform}(2), $\,\LC\,$ has real rank zero.
Lemma \ref{lem:uniform}(3) implies that $\,\{S_\xi\}\in\LC_0^{-1}\,$
whenever $\,\{S_\xi\}\in\LC\,$, $\,S_\xi\in(L_\xi)_0^{-1}\,$ for
each $\,\xi\in\Xi\,$ and
$\,\sup_{\xi\in\Xi}\|S_\xi^{-1}\|<\infty\,$. From here and Lemma
\ref{lem:uniform}(1) it follows that $\,\AC=\{A_\xi\}\in\LC\,$
satisfies the condition {\bf(C)} whenever all the operators
$\,A_\xi\,$ satisfy {\bf(C)}.

Let us fix $\,\eps\in(0,1)\,$ and consider an arbitrary family
$\,\AC=\{A_\xi\}\in\LC\,$ of operators $\,A_\xi\in L_\xi\,$
satisfying {\bf(C)}. Applying Corollary \ref{cor:normal-1} to
$\,\AC\,$, we see that there exist families of operators
$\,\TC_\eps=\{T_{\xi,\eps}\}\in\BR_\LC(\|\AC\|_\LC)\cap\LC_\nR\,$
and $\,\RC_\eps=\{R_{\xi,\eps}\}\in J_{[\AC^*,\AC]}\cap\LC_\sR\,$
such that $\,\|\AC-\TC_\eps-\RC_\eps\|_\LC\leq\eps\,$. The estimate
for the norm holds if and only if
$\,A_\xi-T_{\xi,\eps}-R_{\xi,\eps}\in\BR_\xi(\eps)\,$ for all
$\,\xi\in\Xi\,$. The inclusion
$\,\TC_\eps\in\BR_\LC(\|\AC\|_\LC)\cap\LC_\nR\,$ means that
$\,T_{\xi,\eps}\in\BR_\xi(\|\AC\|_\LC)\cap(L_\xi)_\nR\,$ for all
$\,\xi\in\Xi\,$. Finally, by Remark \ref{rem:M-ideal},
$\,J_{[\AC^*,\AC]}=\bigcup_{t\geq0}tM_{[\AC^*,\AC]}\,$. This
identity and the inclusion $\,\RC_\eps\in
J_{[\AC^*,\AC]}\cap\LC_\sR\,$ imply that $\,R_{\xi,\eps}\in
tM_{[A_\xi^*,A_\xi]}\cap(L_\xi)_\sR\,$ for all $\,\xi\in\Xi\,$ and
some $\,t\,$ independent of $\,\xi\,$. Thus we obtain
\begin{equation}\label{main-3}
A_\xi\ \in\ \BR_\xi(\|\AC\|_\LC)\cap(L_\xi)_\nR
\,+\,tM_{[A_\xi^*,A_\xi]}\cap(L_\xi)_\sR\,+\,\BR_\xi(\eps)\,,
\qquad\forall\xi\in\Xi\,,
\end{equation}
where $\,t\,$ is a nonnegative number which does not depend on
$\,\xi\,$.

If \eqref{main-1} were not true for any $\,h(\eps)\in[0,\infty)\,$
then there would exist families of $\,C^*$-algebras $\,L_\xi\,$ and
operators $\,A_\xi\in L_\xi\,$ satisfying the condition {\bf(C)},
for which \eqref{main-3} would not hold with any $\,t\,$ independent
of $\,\xi\,$. However, by the above, it is not possible. Thus we
have \eqref{main-1} with some function $\,h\,$ for all
$\,\eps\in(0,1)\,$. Since $\,A\in\BR(1)\,$, we can extend
$\,h(\eps)\,$ by zero for $\,\eps\geq1\,$. It remains to notice that
$\,h\,$ can be chosen nonincreasing because the same inclusion holds
for $\,\tilde h(\eps):=\sup_{t\geq\eps}h(t)\,$.
\end{proof}

If $\,t\in[0,\infty)\,$, let
\begin{equation}\label{h0}
F(t)\ :=\ \inf_{\eps>0}\left(h(\eps)\,t+\eps\right)\,,
\end{equation}
where $\,h\,$ is the function introduced in Theorem \ref{thm:main}.
The function $\,F:[0,\infty)\mapsto[0,1]\,$ is nondecreasing,
$\,F(0)=0\,$ and $\,F(t)>0\,$ whenever $\,t>0\,$. Since the subgraph
of $\,F\,$ coincides with an intersection of half-planes, $\,F\,$ is
concave and, consequently, continuous.

\begin{corollary}\label{cor:main-2}
Let $\,L\,$ be a $\,C^*$-algebra of real rank zero, and let
$\,\|\cdot\|_\star\,$  be a continuous seminorm on $\,L\,$ such that
\begin{equation}\label{main-4}
\|USV\|_\star\ \leq\ \|S\|_\star\quad\text{and}\quad\|S\|_\star\leq
C_\star\|S\|\quad\text{for all}\ \,S\in L\,\ \text{and all}\
\,U,V\in L_\uR\,,
\end{equation}
where $\,C_\star\,$ is a positive constant. Then
\begin{equation}\label{main-5}
\inf_{T\in L_\nR\,:\,\|T\|\leq\|A\|}\|A-T\|_\star\ \leq\
C_\star\,F(C_\star^{-1}\|[A^*,A]\|_\star)
\end{equation}
for all operators $\,A\in\BR(1)\,$ satisfying the condition
{\bf(C)}.
\end{corollary}

\begin{proof}
In view of Remark \ref{rem:M-unitary}, from the inequalities \eqref{main-4} it follows that
\begin{equation}\label{main-6}
\|S_1SS_2\|_\star\ \leq\ \|S_1\|\,\|S\|_\star\,\|S_2\|\quad\text{for
all}\ \,S,S_1,S_2\in L
\end{equation}
and, consequently, $\,\|S\|_\star\leq\|[A^*,A]\|_\star\,$ for all
$\,S\in M_{[A^*,A]}\,$. Since $\,\|R\|_\star\leq
C_\star\,\|R\|\leq\eps\,C_\star\,$ for all $\,R\in\BR(\eps)\,$, the
inclusion \eqref{main-1} implies that
$$
\inf_{T\in L_\nR\,:\,\|T\|\leq\|A\|}\|A-T\|_\star\ \leq\
h(\eps)\,\|[A^*,A]\|_\star+\eps\,C_\star\,,\qquad\forall\eps>0\,.
$$
Taking the infimum over
$\,\eps>0\,$, we obtain \eqref{main-5}.
\end{proof}

\begin{remark}\label{rem:convex}
It is clear from the prooof that \eqref{main-5} can be extended to
general functions $\,\|\cdot\|_\star:L\mapsto\R_+\,$ satisfying
\eqref{main-4} and suitable quasiconvexity conditions.
\end{remark}

\begin{example}\label{ex:ideal}
Let $\,J\,$ be a two-sided ideal in $\,L\,$. Then the seminorm
$\,\|A\|_\star:=\dist\left(A,J\right)\,$ satisfies the conditions
\eqref{main-4} with $\,C_\star=1\,$. Corollary \ref{cor:main-2}
implies that
\begin{equation}\label{ideal-1}
\dist\bigl(A,J+\BR(\|A\|)\cap L_\nR\bigr)\ \leq\
F\bigl(\dist([A^*,A],J)\bigr)
\end{equation}
for all $\,A\in\BR(1)\,$ satisfying the condition {\bf(C)}.
\end{example}

\section{Applications}\label{S3}

Throughout this section
\begin{enumerate}
\item[$\bullet$]
$\,\CC(H)\,$ is the $C^*$-algebra of compact operators in $\,H\,$;
\item[$\bullet$]
$\,\SC_p\,$ are the Schatten classes of compact operators and
\item[$\bullet$]
$\,\|\cdot\|_p\,$ are the corresponding norms (we shall always be
assuming that $\,p\geq1$);
\item[$\bullet$]
$\,\|A\|_\ess\ :=\ \inf_{K\in\CC(H)}\|A-K\|\,$ is the distance from
$\,A\,$ to $\,\CC(H)\,$;
\item[$\bullet$]
$\,F\,$ is the function defined by \eqref{h0}.
\end{enumerate}

\subsection{Matrices}\label{S3.matrices}

Let $L$ be the linear space of all complex $\,m\times m$ matrices.
Then the Schatten norms $\,\|\cdot\|_p\,$ on $\,L\,$ satisfy
\eqref{main-4} with $\,C_\star=m^{1/p}\,$. Corollary
\ref{cor:main-2} implies that
\begin{equation}\label{matrix-norm}
\inf_{T\in L_\nR}\|A-T\|\ \leq\ \inf_{T\in
L_\nR\,:\,\|T\|\leq\|A\|}\|A-T\|\ \leq\ F\bigl(\|[A^*,A]\|\bigr)
\end{equation}
and
\begin{equation}\label{matrix-schatten}
\inf_{T\in L_\nR}\|A-T\|_p\ \leq\ \inf_{T\in
L_\nR\,:\,\|T\|\leq\|A\|}\|A-T\|_p\ \leq\
m^{1/p}\,F\bigl(m^{-1/p}\|[A^*,A]\|_p\bigr)
\end{equation}
for all $\,p\in[1,\infty)\,$, all $\,m=1,2,\ldots\,$ and all $\,A\in
L\,$ such that $\,\|A\|\leq1\,$.

Note that the $\,\SC_2$-distance from a given $\,m\times m$-matrix
$\,A\,$ to the set of normal matrices admits the following simple
description.

\begin{lemma}\label{lem:num-range}
Let $\,A\,$ be an $\,m\times m$-matrix, and let $\,\Si_m(A)\,$ be
the set of all complex vectors $\,\zbf\in\C^m\,$ of the form
$$
\zbf\ =\ \{(Au_1,u_1),(Au_2,u_2)\ldots,(Au_m,u_m)\}
$$
where $\,\{u_1,u_2,\ldots,u_m\}\,$ is an orthonormal basis. Then
\begin{equation}\label{matrices-1}
\inf_{T\in L_\nR}\|A-T\|_2^2\ =\ \inf_{T\in
L_\nR\,:\,\|T\|\leq\|A\|}\|A-T\|_2^2\ =\
\|A\|_2^2\;-\sup_{\zbf\in\Si_m(A)}|\zbf|^2\,.
\end{equation}
\end{lemma}

\begin{proof}
If $\,T\,$ is an arbitrary normal matrix and
$\,\{u_1,u_2,\ldots,u_m\}\,$ is the basis formed by its eigenvectors
then
\begin{equation}\label{matrices-2}
\|A-T\|_2^2\ \geq\ \sum_{j\neq k}\left|(A u_j, u_k)\right|^2\ =\ \|
A \|_2^2 - \sum_{j=1}^m\left|(A u_j, u_j)\right|^2\ \geq\
\|A\|_2^2\;-\sup_{\zbf\in\Si_m(A)}|\zbf|^2\,.
\end{equation}
Therefore $\,\inf\limits_{T\in L_\nR}\|A-T\|_2^2\ \geq\
\|A\|_2^2\;-\sup\limits_{\zbf\in\Si_m(A)}|\zbf|^2\,$.

On the other hand, since the set $\,\Si_m(A)\,$ is compact, we have
$\,\sup_{\zbf\in\Si_m(A)}|\zbf|^2=|\zbf_0|^2\,$ for some
$\,\zbf_0\in\Si_m(A)\,$. Let us write down the matrix $\,A\,$ in a
corresponding orthonormal basis $\,\{v_1,\ldots,v_m\}\,$ and denote
by $\,T_0\,$ the normal matrix obtained by removing the off-diagonal
elements. Then $\,\|T_0\|\leq\|A\|\,$ and
$\,\|A-T_0\|_2^2=\sum_{j\neq
k}\left|(Av_j,v_k)\right|^2=\|A\|_2^2-|\zbf_0|^2\,$.
\end{proof}

The following example shows that for $\,p=2\,$ the estimate
\eqref{matrix-schatten} is order sharp as $\,m\to\infty\,$.

\begin{example}\label{ex:h-s}
Let $\,m\,$ be even, and let $\,\{e_j\}_{j=1}^m\,$ be the standard
Euclidean basis in $\,\C^m\,$. Consider the $\,m\times m$-matrix
\begin{equation*}
A = \begin{pmatrix} 0 & 1& 0 & 0 & \dots & 0 & 0 \\
0 & 0 & 0 & 0 & \dots & 0 & 0 \\
0 & 0 & 0 & 1 & \dots & 0 & 0 \\
0 & 0 & 0 & 0 & \dots & 0 & 0 \\
. & . & . & . & \dots & . & . \\
0 & 0 & 0 & 0 & \dots & 0 & 1 \\
0 & 0 & 0 & 0 & \dots & 0 & 0 \end{pmatrix}
\end{equation*}
defined by the identities $\,Ae_{2i}=e_{2i-1}\,$ and
$\,Ae_{2i-1}=0\,$, where $\,i=1,\ldots,m/2\,$. By direct
calculation, $\,\|A\|=1\,$ and
$\,[A,A^*]=\diag(1,-1,\ldots,1,-1,)\,$, so that $\,\|[A,
A^*]\|_2=m^{1/2}\,$.

For this matrix $\,A\,$,
$$
\inf_{T\in L_\nR}\|A-T\|_2\ =\ (m/4)^{1/2}\ =\
m^{1/2}\left(2^{-1}m^{-1/2}\|[A, A^*]\|_2\right)\,.
$$
Indeed, let $\,\{u_j\}_{j=1}^m\,$ be an orthonormal basis
 in $\,\C^m\,$. Then $\,u_j=\sum_{i=1}^{m/2}
(\alpha_{i,j}\,e_{2i-1}+\beta_{i,j}\,e_{2i})\,$, where
$\,\alpha_{i,j}\,$ and $\,\beta_{i,j}\,$ are complex numbers such
that $\,\sum_{i=1}^{m/2}\left( |\alpha_{i,j}|^2+|\beta_{i,j}|^2
\right)= 1\,$. Clearly,
$$
(A u_j, u_j)\ =\ \sum_{i=1}^{m/2}
 \left(\beta_{i,j}\, e_{2i-1}, u_j \right)\
 =\ \sum_{i=1}^{m/2} \beta_{i,j} \, \overline{\alpha_{i,j}}\,.
$$
Therefore, $\,2\left| (A u_j, u_j) \right|\leq\sum_{i=1}^{m/2}
\left(|\alpha_{i,j}|^2+|\beta_{i,j}|^2 \right)=1\,$ and
$\,\sum_{j=1}^m\left|(A u_j, u_j)\right|^2\leq m/4\,$. Thus we have
$\,|\zbf|^2\leq m/4\,$ for all $\,\zbf\in\Si_m(A)\,$. Since
$\,\|A\|_2^2=m/2\,$, Lemma \ref{lem:num-range} implies that
$\,\|A-T\|_2^2\geq m/4\,$ for all normal matrices $\,T\,$. On the
other hand, if $\,T_0=\re A\,$ then $\,\|A-T_0\|_2^2=\|\im
A\|_2^2=m/4\,$.
\end{example}

\begin{remark}\label{rem:lower-bound}
If $\,T\,$ is a normal matrix and $\,A\,$ is an arbitrary matrix of
the same size then
$$
[A^*,A]\ =\ [A^*,T]+[A^*,A-T]\ =\ [(A-T)^*,T]+[A^*,(A-T)]\,.
$$
Estimating the Schatten norms of the right and left hand sides, we
obtain
$$
\|[A^*,A]\|_p\ \leq\ 2\left(\|A\|+\|T\|\right)\|A-T\|_p\,.
$$
This implies that
\begin{equation}\label{lower-bound}
\inf_{T\in L_\nR\,:\,\|T\|\leq\|A\|}\|A-T\|_p\ \geq\
\frac{\|[A^*,A]\|_p}{4\|A\|}
\end{equation}
for all finite matrices $\,A\,$ and all $\,p\in[1,\infty]\,$.
\end{remark}

Substituting the matrix $\,A\,$ from Example \ref{ex:h-s} into
\eqref{lower-bound}, we see that the second estimate
\eqref{matrix-schatten} is order sharp as $\,m\to\infty\,$ for all
$\,p\in[1,\infty)\,$.

\subsection{Bounded and compact operators}\label{S3.bounded-compact}

If $\,L\,$ is the $\,C^*$-algebra obtained from $\,\CC(H)\,$ by
adjoining the unity then, by Remark \ref{rem:l0-1},
$\,L^{-1}=L_0^{-1}\,$ and all $\,A\in L\,$ satisfy the condition
{\bf(C)}. Thus our main results hold for all compact operators
$\,A\,$.

If $\,L=\BC(H)\,$ then $\,L^{-1}=L_0^{-1}\,$ because every unitary
operator can be joined with $\,I\,$ by the path $\,\exp(it\Arg U)\,$
(see Lemma \ref{lem:l0-2}). However, in the infinite dimensional
case $\,\overline{L^{-1}}\ne\BC(H)\,$. The following result was
obtained in \cite{FK} (it also follows from \cite[Theorem 4.1]{CL} or
\cite[Theorem 3]{Bo1}).

\begin{lemma}\label{lem:Bo1} Let $\,H\,$ be separable, and let $\,L=\BC(H)\,$. Then
an operator $\,A\,$ satisfies the condition {\bf(C)} if and only if
for each $\,\la\in\C\,$ either the range $\,(A-\la I)H\,$ is not
closed or $\,\dim\ker(A-\la I)=\dim\ker(A^*-\bar\la I)\,$.
\end{lemma}

In other words, Lemma \ref{lem:Bo1} states that in the separable
case {\bf(C)} is equivalent to the condition on the index function
in the BDF theorem. In particular, this implies that normal
operators and their compact perturbations satisfy the condition
{\bf(C)}.

\begin{remark}\label{Bo2}
An explicit description of the closure of the set of invertible
operators in a nonseparable Hilbert space was obtained in
\cite{Bo2}.
\end{remark}

\subsection{The BDF theorem}\label{S3.bdf}

In this subsection we are always assuming that $\,H\,$ is separable
and $\,L=\BC(H)\,$.

Recall that an operator $\,A\in\BC(H)\,$ is called {\it
quasidiagonal\/} if it can be represented as the sum of a block
diagonal and a compact operator, that is, if there exist mutually
orthogonal finite dimensional subspaces $\,H_k\,$ and operators
$\,S_k:H_k\mapsto H_k\,$ such that $\,H=\bigoplus_{k=1}^\infty
H_k\,$ and $\,A=\diag\{S_1,S_2,\ldots\}+K\,$, where
$\,K\in\CC(H)\,$.

We shall need the following well known result.

\begin{lemma}\label{lem:quasidiag}
The set of compact perturbations of normal operators on a separable
Hilbert space is norm closed and coincides with the set of
quasidiagonal operators $\,S\in\BC(H)\,$ such that
$\,[S^*,S]\in\CC(H)\,$.
\end{lemma}

Lemma \ref{lem:quasidiag} follows from the BDF theorem but it also
admits a simple independent proof based on Theorem \ref{thm:huaxin}
(see \cite[Proposition 2.8]{FR2}). Obviously, the BDF theorem is an
immediate consequence of Corollary \ref{cor:normal-1} and Lemma
\ref{lem:quasidiag}. One obtains a slightly better result by
applying the following lemma, which shows that the BDF theorem holds
with a normal operator $\,T\,$ such that $\,\|T\|\leq\|A\|\,$.

\begin{lemma}\label{lem:quasidiag1}
Let $\,H\,$ be separable, and let $\,L=\BC(H)\,$. Then, for each
fixed $\,r>0\,$, the set $\,B(r)\cap L_n+\CC(H)\,$ is closed and
coincides with the set of quasidiagonal operators
$\,\diag\{S_1,S_2,\ldots\}+K\,$ such that $\,K\in\CC(H)\,$,
$\,S_k\,$ are normal and $\,\|S_k\|\leq r\,$ for all $\,k\,$.
\end{lemma}

\begin{proof}
Obviously, if $\,A=\diag\{S_1,S_2,\ldots\}+K\,$ then $\,A\in
B(r)\cap L_n+\CC(H)\,$ whenever $\,S_k\,$ and $\,K\,$ satisfy the
conditions of the lemma.

Assume now that $\,A\in\overline{B(r)\cap L_n+\CC(H)}\,$. Then
$\,\|A\|_\ess\leq r\,$, $\,[A^*,A]\in\CC(H)\,$ and, by Lemma
\ref{lem:quasidiag}, $\,A=\diag\{S'_1,\,S'_2,\,\ldots\}+K'\,$, where
$\,K'\in\CC(H)\,$ and $\,S'_k\,$ are operators acting in some
mutually orthogonal finite dimensional subspaces $\,H_k\,$ such that
$\,H=\bigoplus_kH_k\,$. Since the self-commutator $\,[A^*,A]\,$ is
compact, we have $\,[(S'_k)^*,S'_k]\to0\,$ as $\,k\to\infty\,$. By
\eqref{matrix-norm}, there are normal operators $\,S''_k:H_k\mapsto
H_k\,$ such that $\,\|S'_k-S''_k\|\to0\,$ as $\,k\to\infty\,$. The
operator $\,\diag\{S'_1-S''_1,\,S'_2-S''_2,\,\ldots\}\,$ is compact,
so that $\,A=\diag\{S''_1,\,S''_2,\,\ldots\}+K''\,$ where
$\,K''\in\CC(H)\,$.

Since $\,\|A\|_\ess\leq r\,$, we have
$\,\limsup_{k\to\infty}\|S''_k\|\leq r\,$. Define
$$
S_k\ :=\
\begin{cases}
S''_k\,,&\text{if}\ \|S''_k\|\leq r\,,\\
r\,\|S''_k\|^{-1}\,S''_k\,,&\text{if}\ \|S''_k\|>r\,.
\end{cases}
$$
Clearly, $\,S_k\,$ are normal and $\,\|S_k\|\leq r\,$. The estimate
for the upper limit implies that
$$
\limsup_{k\to\infty}\|S''_k-S_k\|\ =\
\limsup_{k\to\infty}\left(\|S''_k\|-r\right)_+\ =\ 0\,.
$$
It follows that the operator
$\,\diag\{S''_1-S_1,\,S''_2-S_2,\,\ldots\}\,$ is compact and,
consequently, $\,A=\diag\{S_1,\,S_2,\,\ldots\}+K\,$ where
$\,K\in\CC(H)\,$.
\end{proof}

Theorem \ref{thm:main} and Lemma \ref{lem:quasidiag1} also imply the
following quantitative version of the BDF theorem.

\begin{theorem}\label{thm:bdf1}
Let $\,H\,$ be separable, and let $\,A\in\BC(H)\,$ be an operator
with $\,\|A\|\leq1\,$ satisfying the condition {\bf(C)}.
\begin{enumerate}
\item[(1)]
If $\,[A^*,A]\in\CC(H)\,$ then for each $\,\eps>0\,$ there exists a
diagonal operator $\,T_\eps\in\BC(H)\,$ such that
$\,A-T_\eps\in\CC(H)\,$, $\,\|T_\eps\|\leq\|A\|\,$ and
$\,\|A-T_\eps\|\leq F\bigl(\|[A^*,A]\|\bigr)+\eps\,$.
\item[(2)]
If $\,[A^*,A]\not\in\CC(H)\,$ then for each $\,\eps>0\,$ there
exists a diagonal operator $\,T_\eps\in\BC(H)\,$ such that
$\,\|A-T_\eps\|_\ess\leq 2F\bigl(\|[A^*,A]\|_\ess\bigr)\,$ and
$$
\|A-T_\eps\|\ \leq\ 5F\bigl(\|[A^*,A]\|\bigr)+
3F\Bigl(2F\bigl(\|[A^*,A]\|_\ess\bigr)\Bigr)+\eps\,.
$$
\end{enumerate}
\end{theorem}

\begin{proof}
Since a block diagonal normal operator is represented by a diagonal
matrix in the orthonormal basis formed by its eigenvectors, it is
sufficient to construct a block diagonal normal $\,T_\eps\,$
satisfying the above conditions.

Assume first that $\,[A^*,A]\in\CC(H)\,$. Then, by Corollary
\ref{cor:normal-1} and Lemma \ref{lem:quasidiag1}, we have
$\,A=\diag\{S_1,S_2,\ldots\}+K\,$, where $\,K\in\CC(H)\,$ and
$\,S_k\,$ are normal operators in finite dimensional subspaces
$\,H_k\,$ such that $\,\|S_k\|\leq\|A\|\,$. Let us denote by
$\,E_n\,$ the orthogonal projections onto the subspaces
$\,\bigoplus_{k=1}^nH_k\,$ and define $\,\de_n:=\|K-E_nKE_n\|\,$.
Since
$$
[(E_nAE_n)^*, E_nAE_n]\ =\ E_n(A^*E_nA-AE_nA^*)E_n\ =\
E_n(A^*[E_n,A]+[A^*,A]+A[A^*,E_n])E_n\,,
$$
$\,[E_n,A]E_n=[E_n,K]E_n=(E_nKE_n-K)E_n\,$ and
$\,[A^*,E_n]E_n=(K^*-E_nK^*E_n)E_n\,$, we have
$$
\|\left[(E_nAE_n)^*,E_nAE_n\right]\|\ \leq\
\|[A^*,A]\|+2\de_n\,\|A\|\ \leq\ \|[A^*,A]\|+2\de_n\,.
$$
Applying \eqref{matrix-norm} to the finite rank operators
$\,E_nAE_n\,$, we can find normal operators $\,A_n\,$ acting in
$\,E_nH\,$ such that $\,\|A_n\|\leq\|E_nAE_n\|\leq\|A\|\,$ and
$$
\|E_nAE_n-A_n\|\leq F\bigl(\|[A^*,A]\|+2\de_n\bigr)+\de_n\,.
$$
The block diagonal operators $\,\tilde
T_n:=A_n\oplus\diag\{S_{n+1},\,S_{n+2},\,\ldots\}\,$ are normal,
$\,\|\tilde T_n\|\leq\|A\|\,$ and
$$
A-\tilde T_n\ =\ (E_nAE_n-A_n)+(K-E_nKE_n)\,,\qquad\forall
n=1,2,\ldots
$$
The above identity implies that $\,A-\tilde T_n\in\CC(H)\,$ and
$$
\|A-\tilde T_n\|\ \leq\
F\bigl(\|[A^*,A]\|+2\de_n\bigr)+2\de_n\,,\qquad\forall n=1,2,\ldots
$$
Since $\,K\,$ is compact, $\,\lim_{n\to\infty}\de_n=0\,$ and,
consequently, $\,\lim_{n\to\infty}\|A-\tilde
T_n\|=F\bigl(\|[A^*,A]\|\bigr)\,$. Thus we can take
$\,T_\eps:=\tilde T_n\,$ with a sufficiently large $\,n\,$.

Assume now that $\,\|[A^*,A]\|_\ess>0\,$. From \eqref{ideal-1}
with $\,J=\CC(H)\,$ it follows that
$$
A\ =\ S+K+R\,,
$$
where $\,S\,$ is a bounded normal operator, $\,K\in\CC(H)\,$ and
$\,R\,$ is a bounded operator with $\,\|R\|\leq
2\,F\bigl(\|[A^*,A]\|_\ess\bigr)\,$. Since $\,|F|\leq1\,$, we have
$\,\|S+K\|\leq3\,$.

Let $\,A':=\frac13\,(S+K)\,$. Then $\,[(A')^*,A']\in\CC(H)\,$,
$\,\|A'\|\leq1\,$ and, in view of Lemma \ref{lem:Bo1}, $\,A'\,$
satisfies the condition {\bf(C)}. Applying (1) to $\,A'\,$, we can
find a block diagonal normal operator $\,T'_\eps\,$ and a compact
operators $\,K'_\eps\,$ such that $\,A'=T'_\eps+K'_\eps\,$ and
$\,\|K'_\eps\|\leq F\bigl(\|[(A')^*,A']\|\bigr)+\eps/3\,$. The
identities $\,3A'=S+K=A-R\,$ and the above estimates for
$\,\|A\|\,$, $\,\|R\|\,$ and $\,\|S+K\|\,$ imply that
$$
\|[(A')^*,A']\|\ \leq\ \frac19\left(\|[A^*,A]\|+
2\,\|R\|\,\|S+K\|+2\,\|A\|\,\|R\|\right)\ \leq\
\|[A^*,A]\|+2F\bigl(\|[A^*,A]\|_\ess\bigr)\,.
$$

Obviously, $\,A=3T'_\eps+3K'_\eps+R\,$ and
$\,\|A-3T'_\eps\|_\ess=\|R\|_\ess\leq2F\bigl(\|[A^*,A]\|_\ess\bigr)\,$.
Since $\,\|[A^*,A]\|_\ess\leq\|[A^*,A]\|\,$ and the function $\,F\,$
is nondecreasing and concave,  from the above estimates it follows
that
$$
\|K'_\eps\|\ \leq\
F\Bigl(\|[A^*,A]\|+2F\bigl(\|[A^*,A]\|_\ess\bigr)\Bigr)+\frac\eps3\
\leq\
F\bigl(\|[A^*,A]\|\bigr)+F\Bigl(2F\bigl(\|[A^*,A]\|_\ess\bigr)\Bigr)+\frac\eps3
$$
and, consequently,
$$
\|A-3T'_\eps\|\ \leq\ \|R\|+3\,\|K'_\eps\|\ \leq\
5F\bigl(\|[A^*,A]\|\bigr)+
3F\Bigl(2F\bigl(\|[A^*,A]\|_\ess\bigr)\Bigr)+\eps\,.
$$
Thus we can take $\,T_\eps:=3T'_\eps\,$.
\end{proof}

\begin{remark}\label{rem:diagonolize}
Since $\,\eps\,$ can be chosen arbitrarily small, the distance from
an operator $\,A\,$ satisfying the conditions of Theorem
\ref{thm:bdf1} to the set of diagonal operators does not exceed
$\,5F\bigl(\|[A^*,A]\|\bigr)+
3F\Bigl(2F\bigl(\|[A^*,A]\|_\ess\bigr)\Bigr)\,$. If $\,A\,$ is
normal then this sum is equal to zero and Theorem \ref{thm:bdf1}
turns into the Weyl--von Neumann--Berg theorem for bounded
operators.
\end{remark}

\subsection{Truncations of normal operators}\label{S3.truncations}
Let $\,G\,$ be a positive unbounded self-adjoint operator in a
separable Hilbert space $\,H\,$ whose spectrum consists of
eigenvalues of finite multiplicity accumulating to $\,\infty\,$.
Denote its spectral projections corresponding to the intervals
$\,(0,\la)\,$ by $\,P_\la\,$, and let
$$
N(\la):=\rank P_\la\quad\text{and}\quad
N_1(\la):=\sup_{\mu\leq\la}\left(N(\mu)-N(\mu-1)\right).
$$

If $\,B\in\BC(H)\,$ and $\,[G,B]\in\BC(H)\,$ then, according to
\cite[Theorem 1.3]{LS},
\begin{equation}\label{truncations-1}
\|(I-P_\la)BP_\la\|_2^2\ \leq\ \|(I-P_\la)B(P_\la-P_{\la-1})\|_2^2
+\|(I-P_\la)[G,B](G-\la I)^{-1}P_{\la-1}\|_2^2\,.
\end{equation}
A direct calculation shows that $\,\|(G-\la
I)^{-1}P_{\la-1}\|_2^2\leq\frac{\pi^2}6\,N_1(\la)\,$ (see \cite{LS}
for details). This estimate, \eqref{truncations-1} and the obvious
inequality $\,\|P_\la-P_{\la-1}\|_2^2\leq N_1(\la)\,$ imply that
\begin{equation}\label{truncations-2}
\|(I-P_\la)BP_\la\|_2^2\ \leq\
\left(\|B\|^2+\frac{\pi^2}6\,\|[G,B]\|^2\right)\,N_1(\la)\,.
\end{equation}

Let $\,A\in\BC(H)\,$ be a normal operator such that
$\,[G,A]\in\BC(H)\,$, and let $\,A_\la:=P_\la AP_\la\,$ be its
truncation to the subspace $\,P_\la H\,$. Then
$\,[A^*_\la,A_\la]=P_\la A^*(I-P_\la)AP_\la-P_\la
A(I-P_\la)A^*P_\la\,$ and, consequently,
$$
\|[A^*_\la,A_\la]\|_1\ \leq\ \|P_\la A^*(I-P_\la)AP_\la\|_1 +\|P_\la
A(I-P_\la)A^*P_\la\|_1\,.
$$
Since $\,\|P_\la B^*(I-P_\la)BP_\la\|_1=\Tr(P_\la
B^*(I-P_\la)BP_\la)=\|(I-P_\la)BP_\la\|_2^2\,$, applying
\eqref{truncations-2} with $\,B=A\,$ and $\,B=A^*\,$, we obtain
\begin{equation}\label{truncations-3}
\|[A^*_\la,A_\la]\|_1\ \leq\ C_A\,N_1(\la)\,,
\end{equation}
where $\,C_A:=2\,\|A\|^2+\frac{\pi^2}3\,\|[G,A]\|^2\,$. The
inequalities \eqref{matrix-schatten} and \eqref{truncations-3} imply
that
\begin{equation}\label{truncations-4}
\inf_{T_\la}\|A_\la-T_\la\|_1\ \leq\
N(\la)\,F\bigl(C_A\,N_1(\la)/N(\la)\bigr)\,,
\end{equation}
where the infimum is taken over all normal operators $\,T_\la\,$
acting in the finite dimensional subspace $\,P_\la H\,$.

Assume that there exist positive constants $\,c\,$ and
$\,\varkappa\,$ such that
$\,N(\la)=c\la^\varkappa+o(\la^\varkappa)\,$ as $\,\la\to\infty\,$
(that is, we have a Weyl type asymptotic formula for the counting
function $\,N(\la)\,$). Then $\,N_1(\la)/N(\la)\to0\,$ and,
consequently, $\,F\bigl(C_A\,N_1(\la)/N(\la)\bigr)\to0\,$ as
$\,\la\to\infty\,$. Therefore, in view of \eqref{truncations-4},
there exist normal operators $\,\tilde T_\la\,$ acting in the
subspaces $\,P_\la H\,$ such that $\,\|\la^{-\varkappa}A_\la-\tilde
T_\la\|_1\to0\,$ as $\,\la\to\infty\,$. Roughly speaking, this means
that, under the above conditions on $\,A\,$ and $\,N(\la)\,$, the
normalized truncations $\,\la^{-\varkappa}A_\la\,$ are
asymptotically close to normal matrices with respect to the
$\,\SC_1$-norm.

\begin{remark}\label{rem:truncations-1}
The Weyl asymptotic formula holds for elliptic self-adjoint
pseudodifferential operators on closed compact manifolds and
differential operators on domains with appropriate boundary
conditions. If $\,G\,$ is a pseudodifferential operator of order 1
and $\,A\,$ is the multiplication by a smooth function, as in the
Szeg\"o limit theorem \cite{Sz}, then $\,A\,$ and $\,[G,A]\,$ are
bounded in the corresponding space $\,L_2\,$ and we have
\eqref{truncations-4}.
\end{remark}

\begin{remark}\label{rem:truncations-2}
In \cite{LS}, the classical Szeg\"o limit theorem was extended to
wide classes of self-adjoint (pseudo)differential operators $\,G\,$
and $\,A\,$. More precisely, the authors proved that $\,\Tr
f(A_\la)\sim\Tr P_\la f(A)P_\la\,$ as $\,\la\to\infty\,$ for all
sufficiently smooth functions $\,f:\R\mapsto\R\,$ and all
self-adjoint operators $\,G\,$ and $\,A\,$ satisfying the above
conditions. If $\,f:\C\mapsto\C\,$ and the operator $\,A\,$ is
normal then the right hand side of the above asymptotic formula is
well-defined. However, generally speaking, the truncations
$\,A_\la\,$ are not normal matrices and the left hand side does not
make sense. The results of this subsection suggest that similar
limit theorems can be obtained for (almost) normal operators
$\,A\,$, provided that $\,\Tr f(A_\la)\,$ is understood in an
appropriate sense. For instance, it is plausible that the asymptotic
formula holds for all sufficiently smooth functions
$\,f:\C\mapsto\C\,$ if one defines $\,\Tr
f(A_\la):=\sum_jf(\mu_j)\,$, where $\,\mu_j\,$ are the eigenvalues
of $\,A_\la\,$ (see \cite{Sa}).
\end{remark}

\appendix

\section{Resolution of the identity}\label{A}

The proof of Theorem \ref{thm:partition} is based on successive
reductions of the operator $\,A\in L_\nR\,$ to normal operators
whose spectra do not contain certain subsets of the complex plane.
One can think of this process as removing subsets from $\,\si(A)\,$.
After each step we obtain a new normal operator lying in
$\,L_\nR\,$. The main problem is that, in order to carry on the
reduction procedure, one has to ensure that the removal of a subset
$\,\Om\,$ from the spectrum does not change the spectral projection
corresponding to $\,\C\setminus\overline\Om\,$, and that the new
operator still satisfies the condition {\bf(C)}. In our scheme this
is guaranteed by the equality $\,A(I-\Pi_\Om)=A_\Om(I-\Pi_\Om)\,$
and the condition (a$_4$).

Further on
\begin{enumerate}
\item[$\bullet$] $\,\DC_r(\la)\,$ is the open disc of radius $\,r\,$
centred at $\,\la\in\C\,$, and $\,\partial\DC_r(\la)\,$ is its
boundary;
\item[$\bullet$]
$\,\Sbb:=\partial\DC_1(0)\,$ is the unit circle about the origin.
\end{enumerate}

We shall need the following lemmas which will be proved in the next
two subsections.

\begin{lemma}\label{lem:a1}
Let $\,A\in L_\nR\,$, and let $\,\Pi_\Om\,$ be the spectral
projection of $\,A\,$ corresponding to an open set
$\,\Om\subset\C\,$. If $\,\Om\,$ is homeomorphic to the disc
$\,\DC_1(0)\,$ and $\,A-\mu I\in\overline{L_0^{-1}}\,$ for some
$\,\mu\in\Om\,$ then there exists a normal operator $\,R_\Om:\Pi_\Om
H\mapsto \Pi_\Om H\,$ such that
\begin{enumerate}
\item[(a$_1$)]
$\,(A-R_\Om)\Pi_\Om\in L\,$ and, consequently,
$\,A_\Om:=A(I-\Pi_\Om)\oplus R_\Om\in L_\nR\,$;
\item[(a$_2$)]
$\,\si(R_\Om)\subset\partial\Om\,$, so that
$\,\si(A_\Om)\subset\left(\si(A)\setminus\Om\right)\cup\partial\Om\,$;
\item[(a$_3$)]
$\,A_\Om-\la I\in L_0^{-1}\,$ for all $\,\la\in\Om\,$;
\item[(a$_4$)]
if $\,A\,$ satisfies the condition {\bf(C)} then so does the
operator $\,A_\Om\,$.
\end{enumerate}
\end{lemma}

In other words, we can remove the set $\,\Om\,$ from $\,\si(A)\,$ by
adding a perturbation which does not change $\,A(I-\Pi_\Om)\,$.
Moreover,
\begin{enumerate}
\item[(a$_5$)]
$\,\|A-A_\Om\|=\|(A-R_\Om)\Pi_\Om\|\leq 2r\,$,
\end{enumerate}
where $\,r\,$ is the radius of the minimal disc containing
$\,\Om\,$. This shows that the perturbation is small whenever
$\,\Om\,$ is a subset of a small disc. However, the new operator
$\,A_\Om\,$ may have additional spectrum lying on $\,\partial\Om\,$.

In view of the above, Lemma \ref{lem:a1} is not sufficient for the
study of operators with one dimensional spectra, as does not allow
one to split the one dimensional spectrum into disjoint components.
This problem is resolved by

\begin{lemma}\label{lem:a2}
Let the conditions of Lemma {\rm\ref{lem:a1}} be fulfilled. Assume,
in addition, that
\begin{enumerate}
\item[(i)]
$\,L\,$ has real rank zero;
\item[(ii)]
$\,\si(A)\cap\Om\,$ is a subset of a simple contour $\,\ga\,$ which
intersects $\,\partial\Om\,$ at two points;
\item[(iii)]
$\,A-\mu I\in L_0^{-1}\,$ for all $\,\mu\in\Om\setminus\ga\,$.
\end{enumerate}
Then there exists a normal operator $\,R_\Om:\Pi_\Om H\mapsto
\Pi_\Om H\,$ satisfying the conditions {\rm(a$_1$), (a$_3$),
(a$_4$)} and
\begin{enumerate}
\item[(a$'_2$)]
$\,\si(R_\Om)\subset\ga\cap\partial\Om\,$, so that
$\,\si(A_\Om)\subset\left(\si(A)\setminus\Om\right)\cup\left(\ga\cap\partial\Om\right)\,$
\end{enumerate}
\end{lemma}

\subsection{Proof of Lemma \ref{lem:a1}}\label{A1}
The proof proceeds in three steps.

\subsubsection{}\label{A1.1}

Assume first that $\,\Om=\DC_\eps(0)\,$ and $\,\mu=0\,$, that is,
$\,A\in\overline{L_0^{-1}}\,$. For the sake of brevity, we shall
denote $\,\Pi_\eps:=\Pi_{\DC_\eps(0)}\,$,
$\,R_\eps:=R_{\DC_\eps(0)}\,$ and $\,A_\eps:=A_{\DC_\eps(0)}\,$.

Let $\,A=V\,|A|\,$ be the polar decomposition of $\,A\,$. Let us
consider a sequence of operators $\,B_n\in L_0^{-1}\,$ such that
$\,B_n\to A\,$ as $\,n\to\infty\,$, and let $\,B_n=V_n\,|B_n|\,$ be
their polar decompositions. Then $\,|B_n|\to|A|\,$ and,
consequently, $\,V_n\,|A|\to V\,|A|\,$ as $\,n\to\infty\,$. Since
$\,V_n\,|B_n|=|B^*_n|V_n\,$ and $\,|B_n^*|\to|A^*|=|A|\,$, we also
have $\,|A|V_n\to|A|V\,$ as $\,n\to\infty\,$. It follows that
$\,V_n\,\rho(|A|)\to V\rho(|A|)\,$ and $\,\rho(|A|)V_n\to
\rho(|A|)V\,$ as $\,n\to\infty\,$ for every continuous function
$\,\rho:\R_+\mapsto\R\,$ vanishing near the origin.

Let us fix continuous nonnegative functions $\,\rho_1\,$ and
$\,\rho_2\,$ of $\,\R_+\,$ such that $\,\rho_1\equiv1\,$ and
$\,\rho_2\equiv0\,$ on the interval $\,[\eps,\infty)\,$,
$\,\rho_1\equiv0\,$ near the origin, and
$\,\rho_1^2+\rho_2^2\equiv1\,$. Let
$$
S_n\ :=\ V\rho_1^2(|A|) \;+\;\rho_2(|A|)V_n\,\rho_2(|A|)\,.
$$
The operators $\,S_n\,$ belong to $\,L\,$ because $\,V_n\in L\,$
(see Remark \ref{rem:l-1}), $\,\rho(|A|)\in L\,$ for all continuous
functions $\,\rho\,$, and $\,V\rho_1(|A|)=A\,\tilde\rho_1(|A|)\,$
where $\,\tilde\rho_1(\tau):=\tau^{-1}\rho_1(\tau)\,$ is a
continuous function.

Since $\,V\,$ commutes with $\,|A|\,$, we have
$$
S_n-V_n\ =\
(V-V_n)\,\rho_1^2(|A|)-(V-V_n)\left(I-\rho_2(|A|)\right)\rho_2(|A|)
-\left(I-\rho_2(|A|)\right)(V_n-V)\,\rho_2(|A|)\,.
$$
By the above, the right hand side converges to zero as
$\,n\to\infty\,$. Since $\,V_n\in L_0^{-1}\cap L_\uR\,$ (see Lemma
\ref{lem:l0-2}), this implies that $\,S_n\in L_0^{-1}\,$ for all
sufficiently large $\,n\,$.

Let us fix $\,n\,$ such that $\,S_n\in L_0^{-1}\,$ and consider the
polar decomposition $\,S_n=U_n\,|S_n|\,$. By Lemma \ref{lem:l0-2},
$\,U_n\in L_0^{-1}\cap L_\uR\,$. Since
$\,S_n(I-\Pi_\eps)=V(I-\Pi_\eps)\,$, the operator $\,S_n\,$
coincides with the orthogonal sum $\,V(I-\Pi_\eps)\oplus
S_n\Pi_\eps\,$. The unitary operator $\,U_n\,$ has the same block
structure, $\,U_n=V(I-\Pi_\eps)\oplus U_n\Pi_\eps\,$.

Let $\,R_\eps\,$ be the restriction of $\,\eps\,U_n\,$ to the
subspace $\,\Pi_\eps H\,$. Obviously,
$\,\si(R_\eps)\in\partial\DC_\eps(0)\,$. Since
$\,(A-R_\eps)\Pi_\eps= A-f_\eps(|A|)\,U_n\,$, where
$\,f_\eps(t):=\eps+(t-\eps)_+\,$ is a continuous function, the
operator $\,R_\eps\,$ satisfies the condition (a$_1$).

We have $\,A_\eps=A(I-\Pi_\eps)\oplus R_\eps=f_\eps(|A|)\,U_n\,$,
where $\,f_\eps\geq\eps>0\,$ and $\,U_n\in L_0^{-1}\,$. Therefore
$\,A_\eps\in L_0^{-1}\,$ (see Lemma \ref{lem:l0-2}). Since
$\,\si(A_\eps)\cap\DC_\eps(0)=\varnothing\,$, Lemma \ref{lem:l0-1}
implies that $\,A_\eps-\la I\in \overline{L_0^{-1}}\,$ for all
$\,\la\in\overline{\DC_\eps(0)}\,$.

It remains to prove that $\,A_\eps-\la I\in \overline{L_0^{-1}}\,$
for $\,\la\not\in\overline{\DC_\eps(0)}\,$ whenever $\,A\,$
satisfies the condition {\bf(C)}. Let $\,\Pi_{\de,\la}\,$ be the
spectral projection of $\,A\,$ corresponding to the open disc
$\,\DC_\de(\la)\,$ of radius $\,\de<|\la|-\eps\,$. Applying the
above arguments to the operator $\,A-\la I\,$, we can find an
operator $\,R_{\de,\la}\,$ acting in $\,\Pi_{\de,\la} H\,$ such that
$\,\si(R_{\de,\la})\subset\partial\DC_\de(\la)\,$,
$\,(A-R_{\de,\la})\Pi_{\de,\la}\in L\,$ and
\begin{equation}\label{a1}
A-(A-R_{\de,\la})\Pi_{\de,\la}-\la I\ =\ A(I-\Pi_{\de,\la})\oplus
R_{\de,\la}-\la I\ \in\ L_0^{-1}\,.
\end{equation}
Denote
$$
B_{t,\de}\ :=\ \left((1-t)A+tA_\eps\right)\,(I-\Pi_{\de,\la})\oplus
R_{\de,\la}\,.
$$
Since $\,\Pi_\eps\Pi_{\de,\la}=\Pi_{\de,\la}\Pi_\eps=0\,$ and
$\,A_\eps=A-(A-R_\eps)\Pi_\eps\,$, we have
\begin{equation}\label{a2}
B_{t,\de}\ =\ A(I-\Pi_{\de,\la})\oplus
R_{\de,\la}-t(A-R_\eps)\Pi_\eps\ =\ B_\de^{(1)}\oplus
B_t^{(2)}\oplus R_{\de,\la}\,,
\end{equation}
where $\,B_\de^{(1)}:=A(I-\Pi_{\de,\la})(I-\Pi_\eps)\,$ and
$\,B_t^{(2)}:=\left((1-t)A+tR_\eps\right)\Pi_\eps\,$. Obviously,
$\,\si(B_\de^{(1)})\subset\C\setminus\DC_\de(\la)\,$ and
$\,\si(B_t^{(2)})\subset\overline{\DC_\eps(0)}\,$ for all
$\,t\in[0,1]\,$ (because $\,\|B_{t,\de}^{(2)}\|\leq\eps$). Thus the
spectra of all the operators in the orthogonal sum on the right hand
side of \eqref{a2} do not contain the point $\,\la\,$. Therefore the
operator $\,B_{t,\de}-\la I\,$ is invertible. The first equality
\eqref{a2} implies that $\,B_{t,\de}\in L\,$, so we have
$\,B_{t,\de}-\la I\in L^{-1}\,$ for all $\,t\in[0,1]\,$. By
\eqref{a1}, $\,B_{0,\de}-\la I\in L_0^{-1}\,$ and, consequently,
$\,B_{1,\de}-\la I\in L_0^{-1}\,$. Now, letting $\,\de\to0\,$, we
obtain
$$
\lim_{\de\to0}\left(B_{1,\de}-\la I\right)\ =\
\lim_{\de\to0}\left(A_\eps(I-\Pi_{\de,\la})\oplus R_{\de,\la}-\la
I\right)\ =\ A_\eps-\la I\ \in\ \overline{L_0^{-1}}\,.
$$

\subsubsection{}\label{A1.2}
Let $\,B\in L_\nR\,$, and let $\,\varphi:\C\to\C\,$ be a
homeomorphism isotopic to the identity. The results obtained in
Subsection \ref{A1.1} imply that $\,\varphi(B)\,$ satisfies the
condition {\bf(C)} whenever so does the operator $\,B\,$.

Indeed, let us fix $\,\mu\in\C\,$ and consider the isomorphism
$\,\psi:z\mapsto\varphi(z+\varphi^{-1}(\mu))\,$. Denote
$\,A:=B-\varphi^{-1}(\mu)I\,$, and let $\,A_\eps\,$ be the operators
constructed in Subsection \ref{A1.1}. We have
$\,\mu\not\in\si(\psi(A_\eps))\,$ for all $\,\eps>0\,$ because
$\,\psi^{-1}(\mu)=0\not\in\si(A_\eps)\,$. Moreover, since
$\,\varphi\,$ is isotopic to the identity, the same is true for
$\,\psi\,$ and, by Lemma \ref{lem:l0-1}, $\,\psi(A_\eps)-\mu I\in
L_0^{-1}\,$ for all $\,\eps>0\,$. This implies that
$$
\varphi(B)-\mu I\ =\ \psi(A)-\mu I\ =\
\lim_{\eps\to0}\left(\psi(A_\eps)-\mu I\right)\ \in\
\overline{L_0^{-1}}.
$$

\subsubsection{}\label{A1.3}

Assume now that $\,\Om\,$ is an arbitrary domain and $\,\mu\in\Om\,$
is an arbitrary point satisfying the conditions of the lemma. Let us
fix a homeomorphism $\,\psi:\C\mapsto\C\,$ isotopic to the identity
such that $\psi:\Om\mapsto\DC_1(0)\,$ and $\,\psi(\mu)=0\,$. Denote
$\,\tilde A:=\psi(A)\,$. Then $\,\Pi_\Om\,$ coincides with the
spectral projection of $\,\tilde A\,$ corresponding to the open disc
$\,\DC_1(0)\,$.

By \ref{A1.1}, there exists an operator $\,\tilde R_1\,$ acting in
$\,\Pi_\Om H\,$ such that $\,(\tilde A-\tilde R_1)\Pi_\Om\in L\,$
and $\,\si(\tilde R_1)\subset\partial\DC_1(0)\,$. Let $\,\tilde
A_1:=\tilde A(I-\Pi_\Om)\oplus\tilde R_1\,$ and
$\,R_\Om:=\psi^{-1}(\tilde R_1)\,$. Obviously, the inverse image
$\,R_\Om\,$ satisfies (a$_1$) and (a$_2$), and
$\,A_\Om=\psi^{-1}(\tilde A_1)\,$. Since $\,\psi\,$ is isotopic to
the identity, Lemma \ref{lem:l0-1} implies (a$_3$). Finally, by
\ref{A1.2}, if $\,A\,$ satisfies the condition {\bf(C)} then the
same is true for the operators $\,\tilde A\,$, $\,\tilde A_1\,$ (as
was shown in \ref{A1.1}) and $\,A_\Om\,$. \qed

\subsection{Proof of Lemma \ref{lem:a2}}\label{A2}

It is sufficient to prove the lemma in the case where
$\,\Om=\DC_1(0)\,$ and $\,\ga\cap\Om=(-1,1)\,$. After that, the
general result is obtained by choosing a homeomorphism $\,\psi\,$
isotopic to the identity such that $\,\psi:\Om\mapsto\DC_1(0)\,$ and
$\,\psi:\ga\cap\Om\mapsto(-1,1)\,$ and repeating the same arguments
as in \ref{A1.3}.

Further on we always assume that $\,\Om\,$, $\,\ga\,$ and
$\,\si(A)\,$ are as above and write $\,\Pi_1\,$, $\,R_1\,$ and
$\,A_1\,$ instead of $\,\Pi_\Om\,$, $\,R_\Om\,$ and $\,A_\Om\,$.

\subsubsection{}\label{A2.1}

Suppose first that $\,\si(A)\,$ lies on a simple closed contour
$\,\ga'\,$ homeomorphic to $\,\Sbb\,$.

Let $\,\varphi:\C\mapsto\C\,$ be a homeomorphism isotopic to the
identity such that $\,\varphi:\ga'\mapsto\Sbb\,$ and
$\,\varphi(0)=-1\,$. The operator $\,\varphi(A)\,$ belongs to
$\,L_\uR\,$ because its spectrum lies in $\,\Sbb\,$. The condition
(iii) and Lemma \ref{lem:l0-1} imply that $\,\varphi(A)\in
L_0^{-1}\,$. Therefore, by Lemma \ref{lem:l0-3}, there exist
operators $\,W_n\in L_\uR\,$ such that $\,W_n\to\varphi(A)\,$ as
$\,n\to\infty\,$ and $\,-1\not\in\si(W_n)\,$. Let
$\,B_n:=\varphi^{-1}(W_n)\,$ be their inverse images. Then $\,B_n\,$
belong to $\,L_\nR\,$, $\,\si(B_n)\subset\ga'\setminus\{0\}\,$ for
all $\,n\,$, and $\,B_n\to A\,$ as $\,n\to\infty\,$.

The rest of this subsection is similar to Subsection \ref{A1.1}. Let
us fix continuous nonnegative functions $\,\rho_1\,$ and
$\,\rho_2\,$ of $\,\R_+\,$ such that $\,\rho_1\equiv1\,$ and
$\,\rho_2\equiv0\,$ on the interval $\,[1,\infty)\,$,
$\,\rho_1\equiv0\,$ near the origin, and
$\,\rho_1^2+\rho_2^2\equiv1\,$. Define
$$
\tilde S_n\ :=\ V\rho_1^2(|A|)\;+\;\rho_2(|A|)\,(\re
V_n)\,\rho_2(|A|)
$$
where $\,V\,$ and $\,V_n\,$ are the isometric operators in the polar
representations $\,A=V|A|\,$ and $\,B_n=V_n|B_n|\,$. Note that
\begin{equation}\label{re-Vn}
\rho_2(|B_n|)\,(\re V_n)\,\rho_2(|B_n|)\ =\ (\re
V_n)\,\rho_2^2(|B_n|)\ =\ V_n\,\rho_2^2(|B_n|)
\end{equation}
because $\,\si(B_n)\cap\DC_1(0)\subset(-1,1)\,$ and
$\,\rho_2\equiv0\,$ outside the interval $\,[0,1)\,$.

We have
$$
\tilde S_n-V_n\ =\ (V-V_n)\,\rho_1^2(|A|)+\left(\rho_2(|A|)\,(\re
V_n)\,\rho_2(|A|)-V_n\,\rho_2^2(|A|)\right)\,.
$$
Since $\,\rho_1\equiv0\,$ in a neighbourhood of the origin, the
first term in the right hand side converges to zero. Since
$\,\rho_2(|B_n|)\to\rho_2(|A|)\,$, the identity \eqref{re-Vn}
implies that the second term also converges to zero. Thus
$\,\|\tilde S_n-V_n\|\to0\,$ as $\,n\to\infty\,$ and, consequently,
$\,\tilde S_n\in L^{-1}\,$ for all sufficiently large $\,n\,$.

Let us fix $\,n\,$ such that $\,\tilde S_n\in L^{-1}\,$ and consider
the polar decomposition $\,\tilde S_n=\tilde U_n\,|\tilde S_n|\,$.
The unitary operator $\,\tilde U_n\,$ has the same block structure
as $\,U_n\,$ in the proof of Lemma \ref{lem:a1} but now, in
addition, its restrictions to the subspace $\,\Pi_1 H\,$ is
self-adjoint. Let $\,R_1=\left.\tilde U_n\right|_{\Pi_1H}\,$. Then
$\,R_1\,$ satisfies (a$_1$) and its spectrum can contain only the
points $\,\pm1\,$, so we have (a$'_2$) instead of (a$_2$). By Remark
\ref{rem:l0-1}, $\,A_1\,$ satisfies the condition {\bf(C)}, which
implies (a$_3$) and (a$_4$).

\subsubsection{}\label{A2.2}

 Suppose now that $\,\si(A)\setminus(-1,1)\,$
is an arbitrary subset of $\,\C\setminus\DC_1(0)\,$.

In the process of proof we shall introduce auxiliary operators
$\,A^{(1)}\,$ and $\,A^{(2)}\,$ lying in $\,L_\nR\,$, such that
\begin{enumerate}
\item[($\star$)] the spectral projection of $\,A^{(j)}\,$ corresponding
to $\,\DC_1(0)\,$ coincides with $\,\Pi_1\,$, and
$\,A^{(j)}\,\Pi_1=A\,\Pi_1\,$.
\end{enumerate}
Every next operator will have a simpler spectrum, and $\,R_1\,$ will
be defined in terms of $\,A^{(2)}$.

Let us consider the homotopy $\,\psi_t:\C\mapsto\C\,$ defined by
$$
\psi_t(z)\ =\
\begin{cases}
z\ &\text{if}\ z\in\DC_1(0)\,,\\
(1-t)z+t\,\frac{z}{|z|}\ &\text{if}\ z\not\in\DC_1(0)\,,
\end{cases}\qquad\text{where}\ \,t\in[0,1]\,,
$$
and let $\,A^{(1)}:=\psi_1(A)\,$. Since $\,\psi_1(z)=z\,$ for all
$\,z\in\DC_1(0)\,$ and $\,\psi_1:\C\setminus\DC_1(0)\mapsto\Sbb\,$,
the operator $\,A^{(1)}\,$ satisfies the condition ($\star$) and
$\,\si(A^{(1)})\subset(-1,1)\cup\Sbb\,$. In view of Lemma
\ref{lem:l0-1}, $\,A^{(1)}\,$ also satisfies (iii). Denote by
$\,\tilde\Pi\,$ the spectral projection of $\,A^{(1)}\,$
corresponding to the open lower semicircle $\,\Sbb_-:=\{z\in\Sbb:\im
z<0\}\,$.

Now let us consider the homotopy
$\,\varphi_t:\overline{\DC_1(0)}\mapsto\overline{\DC_1(0)}\,$
defined by
$$
\varphi_t\ =\
\begin{cases}
z-it\,\im z+it\,\sqrt{1-(\re z)^2}\ &\text{if}\ \im z\geq0\,,\\
z+it\,\sqrt{1-(\re z)^2}\ &\text{if}\ \im z\leq0\,,
\end{cases}\qquad\text{where}\ t\in[0,1]\,,
$$
and let $\,\tilde A:=\varphi_1(A^{(1)})$. Since
$\,\varphi_1:\Sbb_-\mapsto(-1,1)\,$,
$\,\varphi_1:(-1,1)\mapsto\Sbb_+\,$ and
$\,\varphi_1:\Sbb_+\mapsto\Sbb_+\,$, the spectrum $\,\si(\tilde
A)\,$ lies on the contour $\,\ga'\,$ formed by the interval
$\,[-1,1]\,$ and the upper semicircle $\,\Sbb_+:=\{z\in\Sbb:\im
z>0\}\,$. By Lemma \ref{lem:l0-1}, the operator $\,\tilde A\,$
satisfies (iii).

Note that $\,\left.\varphi_1\right|_{\Sbb_-}\,$ is a homeomorphism
between $\,\Sbb_-\,$ and $\,(-1,1)\,$. Therefore, the spectral
projection of $\,\tilde A\,$ corresponding to the interval
$\,(-1,1)\,$ coincides with $\,\tilde\Pi\,$. Applying \ref{A2.1} to
the operator $\,\tilde A\,$, we can find a self-adjoint operator
$\,\tilde R\,$ acting in the subspace $\,\tilde\Pi H\,$ such that
$\,(\tilde A-\tilde R)\,\tilde \Pi\in L\,$, $\,\si(\tilde
R)\subset\{-1,1\}\,$ and $\,\tilde A\,(I-\tilde\Pi)\oplus\tilde R\,$
satisfies the condition (iii). Since the restriction of $\,\tilde
A\,$ to $\,\tilde\Pi H\,$ is self-adjoint, we have $\,(\tilde
A-\tilde R)\,\tilde \Pi\in L_\sR\,$.

Let $\,A^{(2)}:=A^{(1)}-\tilde\varphi_1^{-1}((\tilde A-\tilde
R)\,\tilde \Pi)\,$, where
$\,\tilde\varphi_1^{-1}:(-1,1)\mapsto\Sbb_-\,$ is the inverse
mapping $\,\left(\left.\varphi_1\right|_{\Sbb_-}\right)^{-1}\,$.
Then $\,A^{(2)}=A^{(1)}(I-\tilde\Pi)\oplus\tilde R\,$. This implies
that $\,\si(A^{(2)})\subset\ga'\,$ and $\,A^{(2)}\,$ satisfies the
condition ($\star$). Moreover, $\,A^{(2)}\,$ satisfies (iii) because
$$
[0,1]\ni t\ \mapsto\ A^{(1)}-t\tilde\varphi_1^{-1}((\tilde A-\tilde
R)\,\tilde \Pi)-\mu I
$$
is a path in $\,L^{-1}\,$ from $\,A^{(1)}-\mu I\,$ to $\,A^{(2)}-\mu
I\,$ for each $\,\mu\,$ lying in the open domain bounded by
$\,\ga'\,$.

Finally, applying \ref{A2.1} to $\,A^{(2)}\,$, we obtain an operator
$\,R_1\,$ acting in the subspace $\,\Pi_1H\,$ such that
$\,(A^{(2)}-R_1)\,\Pi_1=(A-R_1)\,\Pi_1\in L_\sR\,$ and
$\,\si(R_1)\subset\{-1,1\}\,$. The latter inclusion and (ii) imply
that $\,\si((A-tA+tR_1)\,\Pi_1)\subset[-1,1]\,$ for all
$\,t\in[0,1]\,$. Thus we have
$$
A-t(A-R_1)\,\Pi_1-\mu I\ \in\
L^{-1}\,,\qquad\forall\mu\in\DC_1(0)\setminus(-1,1)\,,\quad\forall
t\in[0,1]\,.
$$
Since the operator $\,A\,$ satisfies (iii), it follows that
$\,A_1-\mu I\in L_0^{-1}\,$ for all $\,\mu\in\DC_1(0)\,$, where
$\,A_1=A(I-\Pi_1)\oplus R_1\,$. Now (a$_3$) and (a$_4$) are proved
in the same way as in Lemma \ref{lem:a1}. \qed

\subsection{Proof of Theorem \ref{thm:partition}}\label{A3}
Every open set $\,\Om_j\,$ coincides with the union of a collection
of open discs. Since the spectrum $\,\si(A)\,$ is compact, it is
sufficient to prove the theorem assuming that $\,\Om_j\,$ is the
union of a finite collection of open disks $\,D_{j,k}\,$. If there
exist mutually orthogonal projections $\,P_{j,k}\,$ such that
$\,\sum_{j,k}P_{j,k}=I\,$ and $\,P_{j,k}H\subset\Pi_{D_{j,k}}H\,$
then we can take $\,P_j:=\sum_kP_{j,k}\,$. Thus we only need to
prove the theorem for open discs $\,\Om_j\,$. In the rest of the
proof we shall be assuming that $\,\Om_j=\DC_{r_j}(z_j)\,$.

The proof is by induction on $\,m\,$. If $\,m=1\,$ then the result
is obvious. Suppose that the theorem holds for $\,m-1\,$ and
consider a family of $\,m\,$ open discs $\,\{\Om_j\}_{j=1}^m\,$
covering $\,\si(A)\,$. If $\,\Om_k\subset\bigcup_{j\ne k}\Om_j\,$
for some $\,k\,$ then we can take $\,P_k=0\,$ and apply the
induction assumption. Further on we shall be assuming that
$\,\Om_k\not\subset\bigcup_{j\ne k}\Om_j\,$ for all
$\,k=1,\ldots,m\,$.

\subsubsection{}\label{A3.1}

If $\,r>t>0\,$, let
$$
\DC_r:=\DC_r(z_m)\quad\text{and}\quad \DC_{t,r}:=
\{z\in\C\,:\,t<|z-z_m|<r\}\,.
$$
Note that
\begin{equation}\label{a3}
\si(A)\setminus\DC_t\ \subset\ \bigcup_{j=1}^{m-1}\,\Om_j
\end{equation}
whenever $\,r_m-t\,$ is small enough. Indeed, if this were not true
then there would exist a sequence of points
$\,\mu_n\in\si(A)\setminus\left(\bigcup_{j=1}^{m-1}\,\Om_j\right)\,$
converging to $\,\partial\Om_m\,$, and the limit point would not
belong to $\,\bigcup_{j=1}^m\Om_j\,$.

In the rest of the proof $\,t\in(0,r_m)\,$ is assumed to be so close
to $\,r_m\,$ that \eqref{a3} holds.

\subsubsection{}\label{A3.2}

In this subsection we are going to construct auxiliary operators
$\,A^{(i)}\in L_\nR\,$ satisfying {\bf(C)} and the following
condition
\begin{enumerate}
\item[($\star\star$)]
$\,\Pi^{(i)}_{\Om_j}H\subset\Pi_{\Om_j}H\,$ for all
$\,j=1,\ldots,m\,$, where $\,\Pi^{(i)}_{\Om_j}\,$ are the spectral
projections of $\,A^{(i)}\,$ corresponding to $\,\Om_j\,$.
\end{enumerate}

Assume first that $\,\partial\Om_m\cap\Om_{j_k}\ne\varnothing\,$ for
some indices $\,j_k\leq m-1\,$. Let us fix an arbitrary
$\,r\in(t,r_m)\,$ and consider the open annulus $\,\DC_{t,r}\,$. The
circles $\,\partial\Om_{j_k}\,$ split $\,\DC_{t,r}\,$ into a finite
collection of connected disjoint open sets $\,\La_\al\,$ such that
$\,\overline{\DC_{t,r}}=\bigcup_\al\overline{\La_\alpha}\,$. Each
set $\,\La_\al\,$ is a circular polygon whose edges are arcs of the
circles $\,\partial\Om_{j_k}\,$, $\,\partial\DC_t\,$ and
$\,\partial\DC_r\,$. Since $\,\Om_k\not\subset\DC_{t,r}\,$ for all
$\,k=1,\ldots,m\,$, the boundaries $\,\partial\La_\al\,$ are
connected and, consequently, each polygon $\,\La_\al\,$ is
homeomorphic to a disc.

Let us remove from $\,\si(A)\,$ the open sets $\,\La_\al\,$,
repeatedly applying Lemma \ref{lem:a1}. Then we obtain an operator
$\,A^{(1)}\in L_\nR\,$ satisfying the condition {\bf(C)}, such that
$$
\si(A^{(1)})\,\subset\,\overline{\DC_t}\cup\left(\cup_\al\,\partial\La_\al\right)
\cup\left(\C\setminus\DC_r\right)\quad\text{and}\quad
\si(A^{(1)})\setminus\overline{\DC_r}=\si(A)\setminus\overline{\DC_r}\,.
$$
Note that $\,\La_\al\subset\Om_j\,$ whenever
$\,\partial\La_\al\cap\Om_j\ne\varnothing\,$. In view of (a$_2$),
this implies that the removal of $\,\La_\al\,$ from the spectrum can
only reduce the eigenspace corresponding to $\,\Om_j\,$. Therefore
$\,A^{(1)}\,$ satisfies the condition ($\star\star$).

Now, repeatedly applying Lemma \ref{lem:a2}, let us remove from
$\,\si(A^{(1)})\cap\DC_{t,r}\,$ the interiors of all edges of the
polygons $\,\partial\La_\al\,$ lying in the open annulus
$\,\DC_{t,r}\,$. Then we obtain an operator $\,A^{(2)}\in L_\nR\,$
satisfying the condition {\bf(C)}, such that
\begin{equation}\label{a4}
\si(A^{(2)})\,\subset\,\overline{\DC_t}\cup\Si
\cup\partial\DC_r\cup\left(\si(A)\setminus\overline{\DC_r}\right)\quad\text{and}\quad
\si(A^{(2)})\setminus\overline{\DC_r}=\si(A)\setminus\overline{\DC_r}\,,
\end{equation}
where $\,\Si\,$ is the set of vertices of the polygons
$\,\La_\al\,$. If at least one point of a closed edge of
$\,\La_\al\,$ belongs to $\,\Om_j\,$, then the interior part of this
edge also lies in $\,\Om_j\,$. In view of (a$'_2$), this implies
that the removal of open arcs does not increase the eigenspaces
corresponding to $\,\Om_j\,$. Therefore $\,A^{(2)}\,$ satisfies
($\star\star$).

By \eqref{a4}, the set of points
$\,z\in\si(A^{(2)})\setminus\DC_r\,$ which do not belong to
$\,\si(A)\setminus\DC_r\,$ consists of a countable collection of
arcs $\,\ga_\beta\,$ of the circle $\,\partial\DC_r\,$, whose end
points belong either to $\,\Sigma\bigcap\partial\DC_r\,$ or to
$\,\si(A)\bigcap\partial\DC_r\,$. Each interior point of
$\,\ga_\beta\,$ is separated from
$\,\si(A^{(2)})\setminus\overline{\DC_r}\,$ (otherwise it would
belong to $\,\si(A)\,$). The set $\,\Sigma\,$ is finite and, by
\eqref{a3}, the intersection $\,\si(A)\bigcap\partial\DC_r\,$ is a
subset of $\,\bigcup_{j=1}^{m-1}\Om_j\,$. This implies that
$\,\left(\si(A^{(2)})\bigcap\partial\DC_r\right)
\setminus\left(\bigcup_{j=1}^{m-1}\Om_j\right)\,$ is covered by a
finite subcollection of arcs $\,\ga_{\beta'}\,$ whose end points
belong to $\,\Sigma\bigcup\left(\bigcup_{j=1}^{m-1}\Om_j\right)\,$.
Repeatedly applying Lemma \ref{lem:a2}, let us remove the interior
parts of the arcs $\,\ga_{\beta'}\,$ from $\,\si(A^{(2)})\,$. Then
we obtain an operator $\,A^{(3)}\in L_\nR\,$ satisfying the
condition {\bf(C)}, such that
\begin{equation}\label{a5}
\si(A^{(3)})\,\subset\,\overline{\DC_t}\cup\Si
\cup\left(\cup_{j=1}^{m-1}\,\Om_j\setminus\DC_r\right)\,.
\end{equation}
For the same reason as before, $\,A^{(3)}\,$ also satisfies the
condition ($\star\star$).

If $\,\partial\Om_m\cap\Om_j=\varnothing\,$ for all
$\,j=1,\ldots,m-1\,$ then $\,\si(A)\,$ is separated from the
boundary $\,\partial\Om_m\,$, and we define $\,A^{(3)}=A\,$.
Obviously, in this case $\,A^{(3)}\,$ also satisfies ($\star\star$)
and \eqref{a5} with $\,\Si=\varnothing\,$ and some $\,t\in(0,r_m)\,$
and $\,r\in(t,r_m)\,$.

\subsubsection{}\label{A3.3}

Let $\,P_m\,$ be the spectral projection of the operator
$\,A^{(3)}\,$ corresponding to the set
$\,\overline{\DC_t}\cup\Si\,$. Since $\,t<r\,$ and $\,\Si\,$ is
finite, this set is separated from $\,\si(A^{(3)})\setminus\DC_r\,$
and, consequently, $\,P_m\in L\,$. Since
$\,\overline{\DC_t}\cup\Si\subset\Om_m\,$, the condition
($\star\star$) implies that $\,P_mH\subset\Pi_{\Om_m}H\,$.

Given $\,z\in\C\,$, let us consider the operator
$$
A_z\ :=\ z\,P_m+(I-P_m)\,A^{(3)}\,.
$$
From \eqref{a5} it follows that
\begin{equation}\label{a6}
\si(A_z)\ \subset\ \{z\}
\cup\left(\cup_{j=1}^{m-1}\,\Om_j\setminus\DC_r\right),\qquad\forall
z\in\C\,.
\end{equation}

If $\,z\in\DC_t\,$ then for each sufficiently small $\,\de>0\,$
there is a homeomorphism $\,\varphi_{z,\de}:\C\mapsto\C\,$ isotopic
to the identity, which maps a neighbourhood of
$\,\overline{\DC_t}\cup\Si\,$ onto $\,\DC_\de(z)\,$ and coincides
with the identity on a neighbourhood of
$\,\si(A^{(3)})\setminus\DC_r\,$. By \ref{A1.2}, all the operators
$\,\varphi_{z,\de}(A^{(3)})\,$ satisfy the condition {\bf(C)}. Since
$\,\varphi_{z,\de}(A^{(3)})\to A_z\,$ as $\de\to0\,$, this implies
that $\,A_z\,$ also satisfy the condition {\bf(C)} for all
$\,z\in\DC_t\,$.

If $\,z\not\in\DC_t\,$ and $\,\la\ne z\,$, let us fix a point
$\,\tilde z\in\DC_t\,$ and a path $\,\mu(s)\,$ from $\,\tilde z\,$
to $\,z\,$ which does not go through $\,\la\,$. Assume that
$\,\eps>0\,$ is so small that $\,\tilde z\not\in\DC_\eps(\la)\,$.
Then, applying Lemma \ref{a1} with $\,\Om=\DC_\eps(\la)\,$ to
$\,A_{\tilde z}\,$, we can find an operator $\,A_{\tilde
z,\eps}:=A_{\tilde z,\Omega}\in L_\nR\,$ such that $\,P_mA_{\tilde
z,\eps}=A_{\tilde z,\eps}P_m=\tilde zP_m\,$, $\,A_{\tilde
z,\eps}-\la I\in L_0^{-1}\,$ and $\,\lim_{\eps\to0}A_{\tilde
z,\eps}=A_{\tilde z}\,$. Since $\,\mu(s)P_m+A_{\tilde
z,\eps}(I-P_m)-\la I\,$ is a path in $\,L^{-1}\,$ from $\,A_{\tilde
z,\eps}-\la I\,$ to $\,zP_m+A_{\tilde z,\eps}(I-P_m)-\la I\,$, the
latter operator also belongs to $\,L_0^{-1}\,$. Therefore
$$
A_z-\la I\ =\ \lim_{\eps\to0}(zP_m+A_{\tilde z,\eps}(I-P_m)-\la I)\
\in\ \overline{L_0^{-1}}\,,\qquad\forall\la\ne z\,.
$$
Obviously, the same inclusion holds for $\,\la=z\,$. Thus the
operators $\,A_z\,$ satisfy the condition {\bf(C)} for all
$\,z\in\C\,$.

\subsubsection{}\label{A3.4}

Let us fix an arbitrary point
$\,z'\in\Om_1\setminus\left(\bigcup_{j=2}^m\Om_j\right)\,$ and
denote $\,A':=A_{z'}\,$. In view of \eqref{a6}, we have
$\,\si(A')\subset\bigcup_{j=1}^{m-1}\Om_j\,$. Applying the induction
assumption to the operator $\,A'\,$, we can find mutually orthogonal
projections $\,P'_1,P_2,\ldots,P_{m-1}\,$ such that
$\,P'_1+\sum_{j=2}^{m-1}P_j=I\,$, $\,P'_1H\subset\Pi'_{\Om_1}H\,$
and $\,P_jH\subset\Pi'_{\Om_j}H\,$ for all $\,j=2,\ldots,m-1\,$,
where $\,\Pi'_{\Om_j}\,$ are the spectral projections of $\,A'\,$
corresponding to $\,\Om_j\,$. Since
$\,z'\not\in\bigcup_{j=2}^{m-1}\Om_j\,$, the projections
$\,\Pi'_{\Om_2},\ldots,\Pi'_{\Om_{m-1}}\,$ coincide with the
spectral projections of the truncation
$\,\left.A^{(3)}\right|_{(I-P_m)H}\,$. Thus we have
$\,P_m\Pi'_{\Om_j}=0\,$ for all $\,j=2,\ldots,m-1\,$ and,
consequently, $\,P_mH\subset P'_1H\,$.

Let $\,P_1:=P'_1-P_m\,$. Then, by the above, $\,P_1,\ldots,P_m\,$
are mutually orthogonal projections such that
$\,\sum_{j=1}^mP_j=I\,$. It remains to notice that, in view of
($\star\star$),
$$
(P'_1-P_m)\,H\ \subset\ \Pi^{(3)}_{\Om_1}H\ \subset\ \Pi_{\Om_1}H
$$
and $\,\Pi'_{\Om_j}H\subset\Pi^{(3)}_{\Om_j}H\subset\Pi_{\Om_j}H\,$
for all $\,j=2,\ldots,m-1\,$.
\qed

\section{Remarks and references}\label{B}

\subsection{}\label{B1}
One can easily show that $\,L_\fR\cap
L_\sR\subset\overline{L^{-1}\cap L_\sR}\,$ and $\,L_\fR\cap
L_\nR\subset\overline{L_0^{-1}\cap L_\nR}\,$ in any $\,C^*$-algebra
$\,L\,$ (see the proof of Corollaries \ref{cor:FR1} and
\ref{cor:FR2}). If $\,L\,$ has real rank zero then
\begin{enumerate}
\item[(1)]
$\,L_\sR=\overline{L_\fR\cap L_\sR}\,$ (this is the implication
$\,(1)\Rightarrow(3)\,$ in Corollary \ref{cor:FR1}) and
\item[(2)]
$\,L_0^{-1}\cap L_\uR=\overline{L_\fR\cap L_\uR}\,$.
\end{enumerate}
The first result is well known and elementary (see, for example,
\cite[Theorem V.7.3]{D} or \cite[Theorem 2.6]{BP}). The second is
due to Huaxin Lin \cite{L1}. Note that (2) is an immediate
consequence of (1) and Lemma \ref{lem:l0-3}.

Using Lemma \ref{lem:partition}, one can deduce from Theorem
\ref{thm:partition} ``quantitative'' versions of (1) and (2), where
the distance to an approximating operator with a finite spectrum
$\,\si\,$ is estimated in terms of $\,\si\,$.

\subsection{}\label{B2}
Theorem \ref{thm:partition} remains valid for self-adjoint operators
$\,A\,$ in a general $\,C^*$-algebra $\,L\,$ satisfying the
condition
\begin{enumerate}
\item[{\bf(C$_\sR$)}]
$\,A-\la I\in\overline{L^{-1}\cap L_\sR}\,$ for all $\,\la\in\R\,$.
\end{enumerate}
Indeed, if we take $\,B_n\in L^{-1}\cap L_\sR\,$ in Subsection
\ref{A1.1} then the operators $\,V_n\,$, $\,S_n\,$ and $\,U_n\,$ are
self-adjoint, and so is the operator $\,A_\eps\,$. The same
arguments show that $\,A_\eps\,$ still satisfies the condition
{\bf(C$_\sR$)}. Therefore, iterating this procedure, we can remove
from $\,\si(A)\,$ an arbitrary finite collection of open intervals
without changing the spectral projections corresponding to the
complements of their closures. This allows us to construct
approximate spectral projections in the same manner as in Subsection
\ref{A3}, with obvious simplifications due to the fact that
$\,\si(A)\subset\R\,$.

Using this observation, one can refine Corollary \ref{cor:FR1} as
follows.

\subsection{}\label{B3} {\it In an arbitrary $\,C^*$-algebra $\,L\,$, the
following statements about a self-adjoint operator $\,A\in L_\sR\,$
are equivalent.
\begin{enumerate}
\item[(1)]
The operator $\,A\,$ satisfies the condition {\bf(C$_\sR$)}.
\item[(2)]
The operator $\,A\,$ has approximate spectral projections in the
sense of Theorem {\rm\ref{thm:partition}}, associated with any
finite open cover of its spectrum.
\item[(3)] $\,A\in\overline{L_\fR\cap L_\sR}\,$.
\end{enumerate}}

As explained in Subsection \ref{B2}, (2) follows from (1),
and the other two implications are proved in the same way as in
Subsection \ref{S2.1}.

\subsection{}\label{B4}
It is clear from the proof that Lemma \ref{lem:a1} remains valid if
we replace $\,L_0^{-1}\,$ with $\,L^{-1}\,$. However, in Lemma
\ref{lem:a2} the assumption (iii) is of crucial importance.

\subsection{}\label{B5}
For a disc $\,\Om=\DC_\eps(0)\,$, Lemma \ref{lem:a1}  without the
condition (a$_4$) can easily be deduced from \cite[Theorem 5]{P}
(see also \cite[Theorem 2.2]{R}). In the both papers the theorem was
proved for $\,A\in\overline{L^{-1}}\,$, but in \cite[Section 3]{FR2}
the authors explained that the approximating operator belongs
$\,\overline{L_0^{-1}}\,$ whenever $\,A\in\overline{L_0^{-1}}\,$.

\cite[Theorem 5]{P} holds if $\,\dist(A,L^{-1})<\eps\,$, whereas we
assumed that $\,\dist(A,L_0^{-1})=0\,$ and, in addition, that
$\,A\,$ is normal. Our proof slightly differs from those in
\cite{P}, \cite{R} and \cite{FR2}. It gives a weaker result in the
general case but is better suited for the study of operators with
one dimensional spectra. It also shows that one can choose
approximating operators satisfying the condition {\bf(C)}.

\subsection{}\label{B6}
Lemma \ref{lem:a2} seems to be new. Possibly, one could deduce the
$\,(1)\Rightarrow(3)\,$ part of Corollary \ref{cor:FR2} from
\cite[Theorem 5.4]{L2}, but our approach gives more information
about the approximating operators. In particular, Theorem
\ref{thm:partition} implies a quantitative (in the same sense as in
Subsection \ref{B1}) version of \cite[Theorem 5.4]{L2}.

\subsection{}\label{B7}
One can further refine Theorem \ref{thm:main} by introducing subsets
$\,M_T^n\subset M_T\,$, which consist of convex combinations of
operators of the form $\,S_1TS_2\,$ containing at most $\,n\,$
terms. The same proof shows that $\,M_{[A^*,A]}\,$ in \eqref{main-1}
can be replaced with $\,M_{[A^*,A]}^{n(\eps)}\,$ where $\,n(\eps)\,$
is an integer-valued nonincreasing function of
$\,\eps\in(0,\infty)\,$.

\subsection{}\label{B8}
A review of results on almost commuting operators and matrices can
be found in \cite{DS}. The authors listed several
dimension-dependent results and discussed the following known
example.

Let $\,A_m\,$ and $\,B_m\,$ be $\,(m+1)\times(m+1)$-matrices defined
by the identities
\begin{enumerate}
\item[]
$\,A_me_j=\left(1-\frac{2j}m\right)e_j\,$ for all
$\,j=0,\ldots,m\,$,
\item[]
$\,B_me_j=\frac2{m+1}\,\sqrt{(j+1)(m-j)}\;e_{j+1}\,$ for all
$\,j=0,\ldots,m-1\,$, and $\,B_me_m=0\,$,
\end{enumerate}
where $\,\{e_0,e_1,\ldots,e_m\}\,$ is an orthonormal basis in
$\,\C^{m+1}\,$. Then $\,\|A_m\|=1\,$, $\,\|B_m\|\leq1\,$,
$\,A_m=A_m^*\,$, $\,\|[B_m^*,B_m]\|\leq\frac4m\,$ and
$\,\|[A_m,B_m]\|\leq\frac2m\,$, so that the Hermitian matrices
$\,A_m\,$, $\,\re B_m\,$ and $\,\im B_m\,$ are almost commuting for
large values of $\,m\,$. However, the distance between the pair
$\,\{A_m,B_m\}\,$ and any pair of commuting
$\,(m+1)\times(m+1)$-matrices is estimated from below by a constant
independent of $\,m\,$ \cite{Ch} (see also \cite{V}).

This example shows that, without additional assumptions,
$\,\BR(\eps)\,$ in \eqref{main-1} cannot be replaced by
$\,\BR(\eps)\cap L_\sR\,$ (or, in other words, it is not sufficient
to adjust only one operator in a pair of almost commuting
self-adjoint operators to obtain a pair of commuting self-adjoint
operators). Indeed, if \eqref{main-1} held with $\,\BR(\eps)\cap
L_\sR\,$ then, applying Theorem \ref{thm:main} to the matrices
$\,\re B_m+iA_m\,$ and $\,\im B_m+iA_m\,$, we could find Hermitian
$\,(m+1)\times(m+1)$-matrices $\,X_m\,$ and $\,Y_m\,$ such that
$\,[A_m,X_m]=[A_m,Y_m]=0\,$ and $\,\|X_m+iY_m-B_m\|\to0\,$ as
$\,m\to\infty\,$.

\subsection{}\label{B9}
Theorem \ref{thm:main} allows one to obtain approximation results
for operators $\,A\in\BC(H)\,$ with compact self-commutators. For
instance, if $\,A\in\BC(H)\,$ satisfies the condition {\bf(C)},
$\,\|A\|\leq1\,$, $\,[A^*,A]\in\SC_p\,$ and $\,\|[A^*,A]\|_p\leq
c\,$ then the number of eigenvalues of each operator from
$\,h(\eps)M_{[A^*,A]}\,$ lying outside the interval
$\,(-\eps,\eps)\,$ does not exceed
$\,\left(c\,\eps^{-1}h(\eps)\right)^p\,$. In view of \eqref{main-1},
this implies that for each $\,\eps>0\,$ there exist a normal
operator $\,T_\eps\,$ and a self-adjoint operator $\,R_\eps\,$ of
finite rank such that $\,\|A-T_\eps-R_\eps\|\leq2\eps\,$ and
$\,\rank R_\eps\leq\left(c\,\eps^{-1}h(\eps)\right)^p\,$. Moreover,
if the operator $\,A\,$ is compact then one can take
$\,T_\eps\in\CC(H)\,$.

Since Theorem \ref{thm:main} does not give an explicit estimate for
$\,h(\eps)\,$, the above observation is of limited interest.
However, it shows that $\,\rank R_\eps\,$ is bounded by a constant
depending only on $\,\eps\,$ and $\,p\,$.

\subsection{}\label{B10}
If $\,A\in B(H)\,$ and $\,\eps>0\,$, let us define
$$
\Spec_\eps(A)\ :=\ \si(A)\bigcup\left\{z\in\C\,:\,\|(A-zI)^{-1}\|>\eps^{-1}\right\}.
$$
The set $\,\Spec_\eps(A)\,$ is called the {\it
$\,\eps$-pseudospectrum} of $\,A\,$. It is known that
$$
\Spec_\eps(A)=\bigcup_{\|R\|<\eps}\si(A+R)\quad\text{and}\quad
\bigcap_{\de>\eps}\Spec_\de(A)=\overline{\Spec_\eps(A)}
$$
(see, for instance, \cite[Theorem 9.2.13]{Da}) and \cite[Lemma
2]{CCH}). Let
$$
d_A(\eps)\,:=
\sup_{\la\in\Spec_\eps(A)}\dist\left(\la,\si_\ess(A)\right)
\quad\text{and}\quad
d_A(0)\,:=\sup_{\la\in\si(A)}\dist\left(\la,\si_\ess(A)\right)\,,
$$
where $\,\si_\ess(A)\,$ is the spectrum of the corresponding element of the Calkin algebra.

In \cite{BD} the authors proved the following statement.
If $\,\|[A^*,A]\|\leq c^2\,$ and
$$
\|(A-\la I)^{-1}\|\ \leq\ \left(\dist\left(\la,\si_\ess(A)\right)-c\right)^{-1},
\qquad\forall\la:\dist\left(\la,\si_\ess(A)\right)>c\,,
$$
then the normal operator $\,T\,$ in the BDF theorem can be chosen in
such a way that $\,\si(T)=\si_\ess(A)\,$ and  $\,\|A-T\|\leq
f(c)\,$, where $\,f:[0,\infty)\to[0,\infty)\,$ is some (unknown)
continuous function vanishing at the origin that depends only on
$\,\si_\ess(A)\,$.

Note that, under the above condition on $\,(A-\la I)^{-1}\,$, we
have $\,d_A(\eps)\leq c+\eps\,$ for all $\,\eps>0\,$. Theorem
\ref{thm:bdf1}(1) implies the following more precise result which
holds without any a priori assumptions about the resolvent.

\subsection{}\label{B11}
{\it Under the conditions of the BDF theorem, there exists a normal
operator $\,T\,$ such that $\,\si(T)=\si_\ess(A)\,$,
$\,A-T\in\CC(H)\,$ and
$$
\|A-T\|\ \leq\ 2\,\|A\|\,F\bigl(\|A\|^{-2}\|[A^*,A]\|\bigr)\ +\
d_A\left(2\,\|A\|\,F\bigl(\|A\|^{-2}\|[A^*,A]\|\bigr)\right)\,,
$$
where $\,F:[0,\infty)\mapsto[0,1]\,$ is a nondecreasing concave
function vanishing at the origin, which does not depend on $\,A\,$.}

Indeed, applying Theorem \ref{thm:bdf1}(1) to the operator
$\,\|A\|^{-1}A\,$, we can find a normal operator $\,T'\,$ such that
$\,A-T'\in\CC(H)\,$ and $\,\|A-T'\|\leq
2\,\|A\|\,F\bigl(\|A\|^{-2}\|[A^*,A]\|\bigr)\,$. By the above,
$\,\si(T')\subset\Spec_\de(A)\,$ for all $\,\de>\|A-T'\|\,$ and,
consequently,
$$
\dist\left(\la,\si_\ess(T')\right)\ =\
\dist\left(\la,\si_\ess(A)\right)\ \leq\
d_A\bigl(\|A-T'\|\bigr)\,,\qquad\forall\la\in\si(T')\,.
$$
Now, using the spectral theorem, one can easily construct a normal
operator $\,T\,$ such that $\,T-T'\in\CC(H)\,$,
$\,\si(T)=\si_\ess(T')\,$ and $\,\|T-T'\|\leq
d_A\bigl(\|A-T'\|\bigr)\,$. This operator satisfies the required
conditions.

\subsection{}\label{B12}

Theorem \ref{thm:main} states that \eqref{main-1} holds for all
$\,C^*$-algebras $\,L\,$ of real rank zero with some universal
function $\,h\,$. The function $\,F\,$ is determined only by $\,h\,$
and, therefore, \eqref{main-5} is true for all $\,C^*$-algebras
$\,L\,$ of real rank zero and all seminorms satisfying the
conditions \eqref{main-4}. Our proof is by contradiction and does
not give explicit estimates for $\,h\,$ and $\,F\,$.

For a particular $\,C^*$-algebra $\,L\,$ and a seminorm
$\,\|\cdot\|_*\,$ on $\,L\,$, it may be possible to optimize the
choice of functions $\,h\,\,$ and $\,F\,$ or to obtain additional
information about their behaviour. Note that
\begin{enumerate}
\item[(i)]
if \eqref{main-1} holds with some function $\,h\,$ then we also have
\eqref{main-5} with $\,F\,$ defined by \eqref{h0} for all seminorms
$\,\|\cdot\|_*\,$ satisfying \eqref{main-4};
\item[(ii)]
$\,\liminf_{\eps\to0}\left(\eps\,h(\eps)\right)>0\,$ for any
function $\,h\,$ satisfying \eqref{main-1} and
\item[(iii)]
$\,\liminf_{t\to0}\left(t^{-1/2}F(t)\right)>0\,$ for any function
$\,F\,$ satisfying \eqref{main-5}
\end{enumerate}
(otherwise we obtain a contradiction by substituting an operator
$\,\de A\,$ and letting $\,\de\to0\,$).

\subsection{}\label{B13}
In \cite{DS} the authors conjectured that the estimate
\eqref{matrix-norm} holds with a function $\,F\,$ such that
$\,F(t)\sim t^{1/2}\,$ as $\,t\to0\,$. In the recent paper
\cite{Ha}, Hastings proved \eqref{matrix-norm} with
$\,F(t)=t^{1/6}\tilde F(t)\,$, where $\,\tilde F\,$ is a function
growing slower than any power of $\,t\,$ as $\,t\to0\,$.

Since the proof of Theorem \ref{thm:bdf1}(1) uses only
\eqref{matrix-norm}, Hastings' result implies the following
corollary.

\subsection{}\label{B14}
{\it Let $\,A\,$ satisfy the conditions of the BDF theorem, and let
$\,\|A\|\leq1\,$. Then for each $\,\eps,\de>0\,$ there exists a
diagonal operator $\,T_{\eps,\de}\,$ such that
$\,A-T_{\eps,\de}\in\CC(H)\,$ and
$$
\|A-T_{\eps,\de}\|\ \leq\ C_\de\,\|[A^*,A]\|^{1/6-\de}\;+\;\eps\,,
$$
where $\,C_\de\,$ is a constant depending only on $\,\de\,$.

}

\subsection{}\label{B15}
In most statements, for the sake of simplicity, we assumed that
$\,\|A\|\leq1\,$. One can easily get rid of this condition by
applying the corresponding result to the operator $\,\|A\|^{-1}A\,$
(as was done in Subsection \ref{B11}).

\end{document}